\documentclass{amsart}
\usepackage{amsmath,amssymb,amsthm,amsfonts, enumitem,appendix}
\usepackage[all]{xypic} 
\usepackage[hidelinks]{hyperref}
\usepackage{cleveref}

\newtheoremstyle{myplain}      {10pt}{10pt}{\itshape}{}{\bfseries}{.}{.5em}{}
\newtheoremstyle{mydefinition} {10pt}{10pt}{}{}{\bfseries}{.}{.5em}{}
\newtheoremstyle{myremark}     {10pt}{10pt}{}{}{\bfseries}{.}{.5em}{}

\theoremstyle{myplain}
\newtheorem{thm}{Theorem}[section]
\newtheorem{lem}         [thm]{Lemma}
\newtheorem{pro}         [thm]{Proposition}
\newtheorem{cor}         [thm]{Corollary}

\theoremstyle{mydefinition}
\newtheorem{defn}    [thm]{Definition}

\theoremstyle{myremark}
\newtheorem{rem}        [thm]{Remark}
\numberwithin{equation} {section}

\DeclareMathOperator{\id}{id}
\let\ker\relax                                        
\DeclareMathOperator{\ker}{Ker}

\newcommand{\Aut}{\operatorname{Aut}}
\newcommand{\Inn}{\operatorname{Inn}}
\newcommand{\Mor}{\operatorname{Mor}}
\newcommand{\emp}{\varnothing}
\newcommand{\Rep}{\operatorname{Rep}}
\newcommand{\redu}{\text{red}}
\newcommand{\kar}{\text{char}}
\newcommand{\ab}{\text{ab}}

\DeclareFontFamily{U}{matha}{\hyphenchar\font45}
\DeclareFontShape{U}{matha}{m}{n}{
      <5> <6> <7> <8> <9> <10> gen * matha
      <10.95> matha10 <12> <14.4> <17.28> <20.74> <24.88> matha12
      }{}
\DeclareSymbolFont{matha}{U}{matha}{m}{n}
\DeclareFontFamily{U}{mathx}{\hyphenchar\font45}
\DeclareFontShape{U}{mathx}{m}{n}{
      <5> <6> <7> <8> <9> <10>
      <10.95> <12> <14.4> <17.28> <20.74> <24.88>
      mathx10
      }{}
\DeclareSymbolFont{mathx}{U}{mathx}{m}{n}

\DeclareMathSymbol{\otop}         {2}{matha}{"6A}

\title[Normal Subgroups, Center and Inner Automorphisms of CQG]{Normal Subgroups, Center and Inner Automorphisms of Compact Quantum Groups}
\author{Issan Patri}
\address{Institute of Mathematical Sciences\\
IV Cross Road, CIT Campus\\
Taramani\\
Chennai 600 113\\
Tamil Nadu, India.}
\email{issanp@imsc.res.in}
\begin{document}
\begin{abstract}
We introduce a class of automorphisms of compact quantum groups which may be thought of as inner automorphisms and explore the behaviour of normal subgroups of compact quantum groups under these automorphisms. We also define the notion of center of a compact quantum group and compute the center for several examples. We briefly touch upon the commutator subgroup of a compact quantum group and discuss how its relation with the center can be different from the classical case.
\end{abstract}

\maketitle

\section{Introduction}
The theory of quantum groups has its roots in the work of G.I. Kac, in his attempt to extend Pontryagin Duality to the case of non-commutative groups. However, it really came to the fore in the 1980's in the pathbreaking work of Drinfeld, Jimbo and Woronowicz, done independently, at roughly around the same time. While Drinfeld \cite{Drin} and Jimbo \cite{Jimbo} constructed deformations of the universal enveloping algebra of simple Lie algebras, the approach of Woronowicz was different. Inspired by Gelfand Duality for commutative $c^\ast$-algebras, Woronowicz, in a series of seminal papers \cite{Woro-PRIMS}\cite{Woro-CMP87}\cite{Woro-Invent}, initiated the study of what are now called compact quantum groups.

Gelfand Duality identifies commutative unital $c^\ast$-algebras as the function algebra of some compact Hausdorff space. This has led to the viewpoint that $c^\ast$-algebras are, in some sense, function algebra of some ``non-commutative'' space. However, function algebras of compact groups fail to remember the group structure, and hence, some additional structure is needed. Woronowicz in \cite{Woro-PRIMS}\cite{Woro-CMP87}\cite{Woro-Invent} identified this extra structure and generalised it to the non-commutative set-up. He also proved the existence and uniqueness of Haar Measure and a Tannaka-Krein type duality theorem in this more general setting, and constructed examples by deforming the compact group $SU(n)$. Several examples have since been constructed \cite{Ban96}\cite{Wang-VD}\cite{Wang-CMP98}.

The notion of subgroups of compact quantum groups was introduced by Podles \cite{Podles} and that of normal subgroups were introduced by Wang \cite{Wang-CMP95}\cite{Wang-JFA}. 

The definition of normal subgroups of compact quantum groups is representation-theoretic, and markedly different from the usual definition of in the classical case of compact groups. This is of course due to a lack of a definition of ``inner'' automorphism. This is one of the motivations of this paper. One of our aims is to define and study a class of automorphisms of a compact quantum group, which in the classical case, can be used to define (closed) normal subgroups. These automorphisms may be thought of as ``inner'' automorphisms. 

In the classical case, the group of inner automorphism is isomorphic to the quotient of the group by its center. This is a motivation to define and study the center of a compact quantum group. Besides, the center, alongwith the commutator subgroup and the identity component, is the most important normal subgroup of a compact group. We briefly study the commutator subgroup and its relationship with the center. The identity component has been studied in great detail in \cite{Pinz}. 

The organisation of the paper and a short description of results are as follows-\\
Section 2 recalls some key notions in the theory of compact quantum groups. Some permanence properties of the notion of co-amenability are also exhibited (Proposition 2.7, Proposition 2.8).\\
Section 3 is devoted to the study of automorphisms of compact quantum groups. We study them from a duality point of view (Proposition 3.2), define a class of automorphisms which can be thought of as ``inner'' automorphisms. This class of automorphisms forms a closed normal subgroup of the group of automorphisms and can be given a suitable topology. We study some topological properties of this subgroup. For instance, it is shown that it is compact for compact quantum matrix groups (Theorem 3.4) and that the quotient of the automorphism group of a compact quantum group by this subgroup gives us a totally disconnected group (Theorem 3.7). We also briefly discuss Hopficity of compact quantum groups.\\
Section 4 covers the stability of normal subgroups of compact quantum groups under the action of these ``inner'' automorphisms. In particular, it is shown that under suitable conditions, a normal subgroup of a compact quantum group is indeed stabilised by any such automorphism (Theorem 4.5).\\
Section 5 studies the converse problem, i.e. the normality of subgroups that are stabilised by all such automorphisms. By means of explicit examples, it is shown that this is not the case in general.\\
Section 6 is devoted to defining the center of a compact quantum group. We start with Wang's notion of central subgroup \cite{Wang-JFA} and derive an equivalent representation-theoretic definition (Theorem 6.3). Then, using ideas of McMullen \cite{Mcmullen}, and using Tannaka-Krein duality, the existence of center is established (Proposition 6.11).\\
Section 7 is devoted to calculation of center of some well-known examples of compact quantum groups. It is shown that compact quantum groups with identical fusion rules have isomorphic center (Theorem 7.1). We also show that the natural generalisation to the quantum group case of the result of M\"uger \cite{Muger} that the chain group, defined for compact groups by Baumg\"artel and Lled\'o \cite{Baum-L}, coincides with the dual of the center of the compact group, is also true.   \\
Finally, in the Appendix, we briefly discuss the commutator subgroup, and by means of an example, show how the relationship between the commutator subgroup and the center of compact quantum matrix groups can be different from that of the commutator subgroup and the center of compact Lie groups.

\section{Preliminaries}
\begin{defn}
  A compact quantum group $G = (A, \Phi)$ is a unital $c^*$-algebra  $A$ together with a comultiplication $\Phi$ which is a coassociative $*$-homomorphism: 
\[\Phi: A \to A \otimes A\]
such that the sets $(A \otimes 1)\Phi(A)$ and $(1 \otimes A)\Phi(A)$ are dense in $A \otimes A$.  
\end{defn}
Now, let $\mathcal{H}$ be a finite dimensional hilbert space with an orthonormal basis given by $\{e_1, e_2, \dots, e_n\}$ and with $e_{ij}$ the corresponding system of matrix units in $B(\mathcal{H})$. A unitary element $u=\sum_{i,j}e_{ij}\otimes u_{ij}$ in $B(\mathcal{H})\otimes A$ is said to be a finite dimensional representation for the compact quantum group $G=(A,\Phi)$ if $\Phi(u_{ij})=\sum_k u_{ik}\otimes u_{kj}$ for all $i,j\in \{1,2,...,n\}$. A finite dimensional representation is said to be irreducible if it has no invariant subspace, see for example \cite{Maes-VD} and \cite{Woro-Notes}.
\begin{defn}
  A compact matrix quantum group $G = (A, \Phi, u)$ is a  triple such that $G=(A, \Phi)$ is a compact quantum group and $u$ is a finite dimensional representation of $G$ such that any irreducible representation is a sub-representation of some tensor power of $u$. 
\end{defn}
It is truly remarkable that a lot of the classical theory of compact groups passes to this more general setting. In particular, one can show the existence of Haar measure and prove Peter-Weyl Theorem. Also, as in the classical case, there is a canonical dense Hopf algebra, denoted by $\mathcal{A}$, generated by the matrix entries of irreducible representations of the compact quantum group $G=(A, \Phi)$. This Hopf algebra is $*$-closed and is equipped with a co-unit $\epsilon_G :\mathcal{A}\rightarrow \mathbb{C}$, which is a $*$-homomorphism satisfying, for all $a\in \mathcal{A}$, 
$$(\epsilon_G\otimes \id)\Phi(a)=(\id\otimes \epsilon_G)\Phi(a)=a$$
It also has an antipode, which is an anti-multiplicative map $\kappa_G:\mathcal{A}\rightarrow \mathcal{A}$ satisfying for all $a\in \mathcal{A}$ 
$$m((\kappa_G\otimes \id)\Phi(a))=m((\id\otimes \kappa_G)\Phi(a))=\epsilon_G(a)\cdot 1$$
where $m:A\otimes A\to A$ denotes the multiplication map, $m(a\otimes b)=ab$ 
\begin{defn}
  A compact quantum group $G=(A, \Phi)$ is said to be coamenable if the Haar measure on it is faithful on $A$ and the counit is norm bounded on $\mathcal{A}$, so can be extended to $A$.  
\end{defn}
While commutative compact quantum groups i.e. compact quantum group with commutative Hopf  C*-algebra, which correspond exactly to the classical case of compact groups, are always coamenable, there are many examples of compact quantum groups which are not co-amenable. For example, the co-commutative quantum group $C^\ast(\Gamma)$, the full group $c^*$-algebra of the discrete group $\Gamma$, is co-amenable if and only if $\Gamma$ is amenable in the group theoretic sense \cite{BMT}.  

Closely related are the notions of maximal and reduced quantum groups associated with a compact quantum group $G = (A, \Phi)$, $G_{\max}=(A_{\max}, \Phi_{\max})$ and $G_{\redu}=(A_{\redu}, \Phi_{\redu})$ respectively. $A_{\max}$ is obtained by a universal property construction, the norm being the supremum over all possible $\ast$-representations of $\mathcal{A}$ into $B(H)$, while $A_{\redu}$ is the $c^{\ast}$-algebra that is the image of $A$ under the GNS map with respect to the Haar measure $h_G$. While $\Phi_{max}$ is the natural extension of $\Phi$ from $\mathcal{A}$ to $A_{\max}$, $\Phi_{\redu}$ is the unique $*-$homomorphism satisfying $$\Phi_{\redu}(\rho_{h_G}(.))=(\rho_{h_G}\otimes\rho_{h_G})\Phi(\cdot)$$ where $\rho_{h_g}$ denotes the GNS representation with respect to $h_G$.

In the coamenable case, $A_{\max} = A = A_{\redu}$. In the cocommutative case, $C^\ast(\Gamma)$ corresponds to the maximal quantum group and $C_r^\ast(\Gamma)$ to the reduced version. Although the underlying $c^\ast$-algebras of $G_{\max}$, $G_{\redu}$ and more generally, any other $c^\ast$-completion of the canonical dense Hopf $\ast$-algebra $\mathcal{A}$  may be different, they can be considered as the same compact quantum group as their representation theory is exactly the same since it is encoded in $\mathcal{A}$. 

We now move to the notion of subgroups introduced by Podles \cite{Podles} and normal subgroups introduced by Wang \cite{Wang-CMP95}\cite{Wang-JFA}. We note that the term subgroup will mean quantum subgroup and will be used interchangeably with it.

\begin{defn}
  A compact quantum group $H = (B, \Psi)$ is said to be a quantum subgroup of $G=(A,\Phi)$ if there exists a surjective $*$-homomorphism $\rho: A \rightarrow B$ such that 
\[(\rho\otimes\rho)\Phi=\Psi \circ \rho\]
This $\rho$ will be called the corresponding surjection.
\end{defn}
Associated with a quantum subgroup $H$ of $G$ are the left coset space and the right coset space given by: 
\begin{align*}
  A_{G/H} &:= \{a \in A \mid (\id \otimes \rho) \Phi(a) = a \otimes 1\}\\
  A_{H\backslash G} &:= \{a \in A \mid (\rho \otimes \id) \Phi(a) = 1 \otimes a\}
\end{align*}
respectively. \\
These spaces have natural conditional expectation onto them given by:
\[E_{G/H}:=(\id \otimes h_H \circ \rho)\Phi\]
and 
\[E_{H\backslash G}:=(h_H \circ \rho \otimes \id)\Phi\]
respectively.

\begin{defn}
A quantum subgroup $N$ of a compact quantum group $G$ is said to be normal if $A_{G/N}=A_{N\backslash G}$. In this case, $G/N=(A_{G/N},\Phi_{|A_{G/N}})$ is a compact quantum group itself.
\end{defn}

\begin{thm}
 [\cite{Wang-JFA}] 
A quantum subgroup $N$ of $G$ is normal if and only if any of the following equivalent conditions hold -- 
\begin{enumerate}
\item $\Phi(A_{G/N}) \subseteq A_{G/N} \otimes A_{G/N}$
\item $\Phi(A_{N\backslash G}) \subseteq A_{N\backslash G} \otimes A_{N\backslash G}$
\item The multiplicity of $1_N$, the trivial representation of $N$ in $\pi_{| N}$, where $\pi$ is any irreducible representation of $A$, is either $0$ or $d_\pi$, the dimension of $\pi$.  
\end{enumerate}
\end{thm}
This generalises the classical case of compact groups \cite{Hew-Ross} and in the cocommutative case, it can be shown that any subgroup is in fact normal. 
\begin{pro}
  If $G=(A,\Phi)$ has a coamenable normal subgroup $N=(B,\Psi)$ such that $G/N$ is coamenable, then  $G$ is coamenable. 
\end{pro}
\begin{proof}
Let $\rho:A\rightarrow B$ be the corresponding surjection. It follows from the uniqueness of the co-unit $\epsilon_G$ that $\epsilon_G=\epsilon_N\circ \rho$ and so, the co-unit is norm bounded on $\mathcal{A}$.\\
To show that the Haar measure $h_G$ is faithful, we first note that the conditional expectation onto $N\backslash G$ given by $$E_{N\backslash G}=(h_H\circ \rho\otimes \id)\Phi$$
is faithful since $(\rho\otimes \id)\Phi$ is injective as $$(\epsilon_N\circ \rho\otimes \id)\Phi(a)=(\epsilon_G\otimes \id)\Phi(a)=a$$
and $h_H$ is faithful. But then as $E_{N\backslash G}$ is invariant under $h_G$, it follows that $h_G$ is faithful. 
\end{proof}
\begin{pro}
  Let $G$ be a coamenable compact quantum group. Then, any subgroup $H$ of $G$ is coamenable as well. 
\end{pro}
\begin{proof} We use a little bit of machinery for this. It is shown in \cite{Kyed} that a compact quantum group is coamenable if and only if its fusion algebra is amenable in the sense of \cite{Hiai-Iz}. The proposition now follows from Proposition 7.4(2) of \cite{Hiai-Iz}. \end{proof}

We now turn to Tannaka-Krein Duality $^\ast$. In a seminal paper \cite{Woro-Invent}, Woronowicz proved a version of Tannaka-Krein Duality for compact quantum groups. For $G=(A,\Phi)$ a compact quantum group, let $\Rep(G)$ denote the representation category of $G$, a category whose objects are finite dimensional representations of $G$ and whose arrows are defined by: 
\[(u,v) :=\{T \in B(\mathcal{H}_u, \mathcal{H}_v) \mid (T \otimes 1) u = v ( T \otimes 1)\}\]
i.e. the intertwiner space of the two representations. This category is a tensor $c^\ast$-category. Woronowicz encoded this information in a triple  
\begin{equation}(R, \{\mathcal{H}_\alpha\}_{\alpha \in R}, \{\Mor(\alpha, \beta)\}_{\alpha,\beta \in R})\end{equation}
where $R$ is a monoid, given any $\alpha\in R$, $\mathcal{H}_\alpha$ is a finite dimensional Hilbert space and given $\alpha,\beta\in R$, $\Mor(\alpha, \beta) \subseteq B(H_\alpha, H_\beta)$ is a subspace. Given a compact quantum group $G$, such a triple is constructed by taking the set of all finite dimensional representations of $G$, which can be given a monoid structure by tensor product, with the trivial representation as the unit element, $\mathcal{H}_\alpha$ denoting the ambient hilbert space of the representation $\alpha$ and $\Mor(\alpha, \beta)$ being the intertwiner space of the representation $\alpha$ and $\beta$. We call such a triple the quantum group triple associated with $G$.

Woronowicz proved that any triple as in (2.1) is a quantum group triple for some suitable compact quantum group denoted by $R_{\max}$. This compact quantum group is maximal. Its canonical dense Hopf $\ast$-algebra will be denoted $\mathcal{A}_R$ 

\section{Automorphisms and Inner Automorphisms}
Let $G = (A, \Phi)$ be a compact quantum group. A quantum group automorphism (henceforth, the word automorphism will always mean a quantum group automorphism) is a $c^\ast$-algebraic automorphism $\alpha: A \to A$ such that $(\alpha \otimes \alpha)\Phi = \Phi \circ \alpha$. 

\begin{pro}
  Let $\alpha$ be an automorphism of $G = (A, \Phi)$. Then, 
  \begin{enumerate}[font = \upshape]
  \item $\alpha(\mathcal{A})=\mathcal{A}$
  \item $\epsilon_G \circ \alpha = \epsilon_G$. 
  \item $\kappa_G \circ \alpha = \alpha \circ \kappa_G$ 
  \item $h_G \circ \alpha = h_G$ 
  \item If $((u_{ij})) \in M_n(A)$ is a finite dimensional irreducible representation of $G$, then so is $((\alpha(u_{ij}))) \in M_n(A)$
  \end{enumerate}
\end{pro}
\begin{proof}\mbox{}
  \begin{enumerate}
  \item Let $((u_{ij})) \in M_n(A)$ be a unitary representation of $G$. Then, it is easy to see that $((\alpha(u_{ij}))) \in M_n(A)$ is also a representation. It then follows that $\alpha(\mathcal(A)) \subseteq \mathcal(A)$. But, it is also easily checked that $\alpha^{-1}$ is an automorphism of $G$ and hence $\alpha(\mathcal{A}) = \mathcal{A}$. 
  \item Follows from the uniqueness of the counit $\epsilon_G$ under the condition that $(\epsilon_G \otimes \id) \Phi = (\id \otimes \epsilon_G) \Phi = \id$.  \item Follows from the uniqueness of the antipode under the condition that $m(\kappa_G \otimes \id) \Phi(\cdot) = m(\id \otimes \kappa_G) \Phi (\cdot) = \epsilon_G(\cdot) 1$. 
  \item Follows from the uniqueness of Haar measure $h_G$ under the condition that $(h_G \otimes \id)\Phi(a) = (\id \otimes h_G)\Phi(a) = h_G(a) 1$
  \item Let $u = ((u_{ij})) \in M_n(A)$ be an finite dimensional representation. Then, $\chi_u = \sum u_{ii}$ is the character of the representation $u$. It is shown in \cite{Woro-CMP87} that $u$ is irreducible if and only if $h_G(\chi_u^\ast \chi_u) = 1$. The result now follows immediately from part (4).    
  \end{enumerate}
\end{proof}
Given a compact quantum group $G = (A, \Phi)$, consider the associated triple, $$(R_G, \{\mathcal{H}_a\}_{a \in R_G}, \{\Mor(a, b)\}_{a,b\in R_G})$$ Any automorphism $\alpha: A \to A$ of $G$, gives a monoidal automorphism, $\widehat{\alpha}: R \to R$ such that 
\begin{enumerate}
\item $\mathcal{H}_a = \mathcal{H}_{\widehat{\alpha}(a)},\;\forall\; a \in R_G$ 
\item $\Mor(a, b) = \Mor ( \widehat{\alpha}(a), \widehat{\alpha}(b))\;\forall\; a,b\in R_G$
\end{enumerate}
In fact, the converse is also true, given a triple $$(R, \{\mathcal{H}_{a}\}_{a \in R}, \{\Mor(a,b)\}_{a,b\in R})$$ with a monoidal automorphism $\widehat{\alpha}: R \to R$ such that $\mathcal{H}_{\widehat{\alpha}(a)}=\mathcal{H}_a$ for all $a \in R$ and $\Mor(\widehat{\alpha}(a), \widehat{\alpha}(b))=\Mor(a,b)$, we can construct an automorphism $\alpha: R_{\max} \to R_{\max}$. 
\begin{pro}
Let $(R, \{\mathcal{H}_a\}_{a \in R}, \{\Mor(a,b)\}_{a,b\in R})$ be a quantum group triple as above. Then, a monoidal automorphism $\widehat{\alpha}: R \to R$ induces an automorphism $\alpha: R_{\max} \to R_{\max}$ if and only if 
\begin{enumerate}
\item $\mathcal{H}_a = \mathcal{H}_{\widehat{\alpha}(a)}\; \forall\; a \in R$ 
\item $\Mor(\widehat{\alpha}(a), \widehat{\alpha}(b)) = \Mor(a,b)\; \forall\; a, b \in R$. 
\end{enumerate}
\end{pro}
\begin{proof}
  ($\Rightarrow$): Easy to see. \\
  ($\Leftarrow$): We refer to \cite{Woro-Invent} for definition of unexplained terms. It is easy to see that $\{R_{\max}, u_{\widehat{\alpha}(a)}\}$ gives a model for $R$. It now follows from Proposition 3.3 of \cite{Woro-Invent} that there is a map $\alpha: \mathcal{A}_R \to R_{max}$ such that $$\alpha(u_a) = u_{\widehat{\alpha}(a)}$$ Since $\mathcal{A}_R$ is dense in $R_{max}$, it can be extended to the whole of $R_{\max}$, since $R_{\max}$ is obtained by a maximal norm construction (see for example the discussion at the end of Section 3 of \cite{Woro-Invent}). It is an automorphism as $\widehat{\alpha}$ is an automorphism, so the same can be done for $\widehat{\alpha}{^{-1}}$ and since $\{u_{a}\}_{a\in R}$ generate $\mathcal{A}_R$, it is easy to check that $(\alpha\otimes\alpha)\Phi=\Phi\circ \alpha$
\end{proof}
Some authors deal with automorphisms or more generally homomorphisms at a purely algebraic level, since given an automorphism $\alpha: A  \to A$ of a compact quantum group $G = (A, \Phi)$, it gives an automorphism, by restriction, $\alpha: \mathcal{A} \to \mathcal{A}$, of the Hopf $\ast$-algebra $\mathcal{A}$. Now let $\mathcal{A}$ be the canonical dense Hopf $\ast$-algebra of some compact quantum group $G$ (such Hopf $\ast$-algebras are called CQG algebras \cite{Dij-K}). Suppose $\underline{\alpha}$ is an $\ast$-automorphism of $\mathcal{A}$ which commutes with the co-product. Then $\underline{\alpha}$ can be extended, by universality, to give an automorphism $\alpha$ of the compact quantum group $G_{\max}$.

But an automorphism of a compact quantum group is also Haar-measure preserving, so in particular, $\alpha$ being an automorphism of $G_{max}$, will be  implemented by some unitary $$u_\alpha: L^2(A_{\max}, h_{G_{\max}}) \to L^2(A_{\max}, h_{G_{\max}})$$ and so, with $A_{\redu} \subseteq \mathcal{B}(L^2(A_{\max}, h_{G_{\max}})$ we have, 
\begin{align*} 
\alpha_{\redu}: A_{\redu} &\to A_{\redu} \\
\alpha_{\redu}(a) &= u_\alpha a u_\alpha^{\ast}
\end{align*}
 and 
$$\alpha_{\redu|{\mathcal{A}}}=\underline{\alpha}$$. 

So, given an automorphism of a CQG Algebra $\mathcal{A}$, we can extend it to get automorphism of both $G_{\max}$ and $G_{\redu}$, the associated maximal and reduced compact quantum groups. Since these are the two main types of compact quantum groups that we will be considering, and also, since $c^\ast$-algebraic automorphisms are interesting in their own right, our automorphisms will always be $c^\ast$-algebraic. 

In the classical case, our definition of an automorphism reduces to continuous group automorphisms. In the cocommutative case, with the underlying discrete group being $\Gamma$, any automorphism gives us a group automorphism of $\Gamma$ and of course, the converse is easily seen, i.e., any group automorphism of $\Gamma$ extends to an automorphism of both $C^\ast(\Gamma)$ and $C^\ast_{\redu}(\Gamma)$.

One interesting consequence of this follows from a lemma of Wang, which we recall: 
\begin{lem}
 [\cite{Wang-JFA}]
Let $A, A_1, A_2$ be $c^\ast$-algebras, $\pi_1: A \to A_1$ and $\pi_2: A \to A_2$ be surjections with kernels $I_1$ and $I_2$ respectively and $P_1$ and $P_2$ be the pure state space of $A_1$ and $A_2$ respectively. Then, the following are equivalent: 
\begin{enumerate}[font = \upshape]
\item $\{\phi_1\circ\pi_1 : \phi_1\in P_1\}=\{\phi_2\circ\pi_2:\phi_2\in P_2\}$ as subsets of pure states of $A$.
\item $I_1=I_2$
\item There is an isomorphism $\alpha:A_1\rightarrow A_2$ such that $\pi_2=\alpha\circ \pi_1$.
\end{enumerate}
\end{lem}
Now let $\mathcal{A}$ be a CQG algebra, with $G_{\max}$ and $G_{\redu}$ the corresponding maximal and reduced compact quantum groups. In the set up of the previous lemma, let $A = A_{\max}$ and $A_1 = A_{\redu}=A_2$. If $\alpha$ is an automorphism of $G_{\max}$, then we have the following diagram: 
\begin{center}
$\xymatrix{
A_{\max} \ar[d]_{\psi} \ar[r]^{\alpha} & A_{\max} \ar[d]^{\psi} \\
A_{\redu} \ar[r]_{\alpha_{\redu}} & A_{\redu} }$
\end{center}
Here $\psi$ denotes the canonical surjection from $A_{\\max}$ onto $A_{\redu}$. Since $\left.\alpha\right|_{\mathcal{A}}=\left.\alpha_{\redu}\right|_{\mathcal{A}}$, this diagram commutes, so $\ker(\psi)=\ker(\psi \circ \alpha)$ and hence, $\alpha(I_\psi)=I_\psi$, where $I_\psi$ denotes $\ker(\psi)$. So, every such automorphism keeps the ideal $I_\psi$ stable. 

In the classical case, a special class of automorphisms are the inner automorphisms, the automorphisms of the form $\alpha_s: G \to G$ where
\[\alpha_s(g) = sgs^{-1},\;\; s,g \in G\]
Let's note that the inner automorphisms preserve the class of any unitary representation of the compact group. In other words, the fixed point algebra of $C(G)$ under the action of $G$ on it by inner automorphisms contains the characters of all irreducible representations of the group, and in fact, it follows from Peter-Weyl theorem that the linear span of characters of all irreducible representations is dense in this algebra. 

Let's consider the more general class of automorphisms that preserve the representation class of each irreducible representation. Denoting this subset of the automorphism group by $\Aut_\chi(G)$, it is easily seen that this gives a normal subgroup of the group $\Aut(G)$. It is known for compact connected groups that $\Aut_\chi(G)=\Inn(G)$, where $\Inn(G)$ denotes the inner automorphisms of $G$ \cite{Mcmullen}. But there are several examples of finite groups for which $\Inn(G)$ is a proper subgroup of $\Aut_\chi(G)$ (for more on this we refer to the survey \cite{Yadav} and to references therein). 

Given a compact quantum group $G = (A, \Phi)$, let $G_{\kar}$ denote the set of characters of irreducible representations of $G$. We want to consider the group of automorphisms 
\[\Aut_\chi(G) = \{\alpha \in \Aut(G) \mid \alpha(\chi_{u^a}) = \chi_{u^a} \; \forall \; \chi_{
u^a} \in G_{\kar} \} \]

It is straightforward to see, using Proposition 3.1(5), that $\Aut_{\chi}(G)$ is a normal subgroup of $\Aut(G)$.

We topologise $\Aut(G)$ by taking as neighbourhood of identity, sets of the form:
\[u(a_1, \dots, a_n \in A \mid \epsilon > 0) := \{\alpha \in \Aut(G) \mid \|a_i - \alpha(a_i) \| < \epsilon \;\forall i \in \{1,2,\dots, n\}\}\]
$\Aut_\chi(G)$ is easily seen to be a closed normal subgroup. 
For the next theorem, we assume that $G=(A,\Phi,u)$ is compact matrix quantum group.
\begin{thm}
 $\Aut_\chi(G)$ is a compact group. 
\end{thm}
\begin{proof}
  Viewing $u$ as an element of $M_n(A) = M_n(\mathbb{C}) \otimes A$, let $u = ((u_{ij}))$. Let $\alpha \in \Aut_\chi(G)$, then we have: 
\[\sum u_{ii} = \sum \alpha(u_{ii})\]
so that there exists a $u_\alpha \in U(n)$, the group of $n\times n$ unitary (scalar) matrices. such that:
\[(u_\alpha \otimes 1) u (u_\alpha^\ast \otimes 1) = \alpha(u)\]
so, this gives us an anti-homomorphism, \[\gamma: \Aut_\chi(G) \to U(n)\]\[\alpha \mapsto u_\alpha\] where $\alpha(u) = (u_\alpha \otimes 1) u (u_\alpha^\ast \otimes 1)$. Clearly $\gamma$ is injective. 

To show $\Aut_\chi(G)$ is compact, we want to show that the image of $\gamma$ in $U(n)$, denoted by $\Im(\gamma)$, is closed in $U(n)$ and the map $\gamma^{-1}: \Im(\gamma) \to \Aut_\chi(G)$ is continuous.    
We have a lemma: 
\begin{lem}
  Let $\{\alpha_i\}_{i\in I}$ be a net of automorphisms of a $c^\ast$-algebra $A$ which is generated as a $c^\ast$-algebra by the set $\{s_1,..., s_n\}_{n\in \mathbb{N}}$. If $\alpha_i(s_k) \to s_k$ for all $k \in \{1, \dots, n\}$ in norm, then, $\alpha_i(a) \to a$ for all $a \in A$ in norm.  
\end{lem}
\begin{proof}
  This is straightforward.
\end{proof}
It now follows from the previous lemma that $\gamma^{-1}$ is continuous since $A$ is generated by $u_{ij}$'s, the matrix entries of the representation $u = ((u_{ij}))$. 

To show that $\Im(\gamma)$ is closed, we need the following lemma, a matrix version of a lemma of Pisier \cite{Pisier}, which is proved in \cite{Ar-P}. We sketch the proof here for the convenience of the reader. 
\begin{lem}
  Let $A$ be a unital $c^\ast$-algebra and $u = ((u_{ij})) \in M_n(A)$ be an unitary element, such that $A$ is generated as a $c^\ast$-algebra by $u_{ij}, i,j\in \{1,2,...,n\}$ and suppose $T: A \to B$ is a unital completely positive map into some $C^\ast$-algebra $B$ such that $((T(u_{ij}))) \in M_n(B)$ is also a unitary element. Then, $T$ is a $\ast$-homomorphism.  
\end{lem}
\begin{proof}
  The proof is by a multiplicative domain argument. We want to show that $u_{ij}$'s are in the multiplication of domain of $T$ from which the result will follow.

Let us take $u_{11} \in A$. Then, we know that 
\[\sum u_{1j} u_{1j}^\ast = 1.\] 
Now, by Cauchy-Schwarz inequality, 
\[T(u_{11} u_{11}^\ast) \geqslant T(u_{11}) T(u_{11})^\ast\]  
But, we also have, 
\begin{align*}
  T(u_{1j}u_{1j}^\ast) &\geqslant T(u_{1j})T(u_{1j})^\ast \\
 \Rightarrow \sum_{j=2}^n T(u_{1j}u_{1j}^\ast) &\geqslant \sum_{j=2}^n T(u_{1j})T(u_{1j})^\ast \\
 \Rightarrow T(1 - u_{11}u_{11}^\ast) &\geqslant 1 - T(u_{11})T(u_{11}^\ast) \\
 \Rightarrow T(u_{11} u_{11}^\ast) &= T(u_{11}) T(u_{11})^\ast
\end{align*}
Similarly, this can be proved for all $u_{ij}$'s. 
\end{proof}

We continue with the proof of the theorem. Let $\{t_i\}_{i\in I}$ be a net of unitary matrices in $U(n)$ such that $t_i \to t$ in $U(n)$, with $\gamma (\alpha_i)=t_i\: \forall i\in I$. Consider the finite dimensional operator system $S$ generated by $\{u_{ij}\}, i,j\in\{1,2,...,n\}$, the matrix entries of $u \in M_n(A)$. Since $\alpha_i$ are automorphisms such that:
\[\alpha_i(u) = (t_i \otimes 1) u (t_i^\ast \otimes 1)\]
the map on $S$ defined by $$1 \mapsto 1$$ and $$u \mapsto (t_i\otimes 1) u (t_i^\ast\otimes 1)$$ gives us an unital completely positive map $\phi_i: S \to S$ for all $i\in I$. Now, consider the map $\phi: S \to S$ given by $$1 \mapsto 1$$  $$u \mapsto (t \otimes 1) u (t \otimes 1)^\ast$$ Since $t_i \to t$ in norm, it follows that $\phi_i \to \phi$ in point norm topology. This implies $\|\phi\|_{cb} = 1$ and since $\phi$ is unital, we have that $\phi$ is unital completely positive. 

But by previous lemma, $\phi$ extends uniquely to an automorphism of the quantum group and so range of $\gamma$ is closed and we are done.     
\end{proof}

We now revert back to our original assumption of $G$ being a compact quantum group. 

\begin{thm}
  The group $\mathrm{Out}_{\chi} (G)=\Aut(G)/ \Aut_\chi(G)$ is totally disconnected. 
\end{thm}

\begin{proof}
  Let $u$ be an irreducible representation of $G$ and $\chi_u$ be its character, $\chi_u \in A$. Let 
\[K_{[u]} := \{\alpha \in \Aut(G) \mid \alpha(\chi_u) = \chi_u\}\] 
Then we have, 
\[\Aut_\chi(G) = \bigcap_{[u] \in \widehat{G}} K_{[u]}\]
where $\widehat{G}$ denotes the set of equivalence classes of irreducible representations of $G$ and $[u]$ the equivalence class corresponding to the irreducible representation $u$ of $G$. We shall show that $K_{[u]}$ is an open subgroup of $\Aut(G)$ which will imply that $\Aut(G)/ K_{[u]}$ is discrete and so it will follow that $ \Aut(G)/\Aut_\chi(G)$ is totally disconnected. 

To show that $K_{[u]}$ is open, consider the open neighbourhood: 
\[u(\chi_u, 1) := \{\alpha \in \Aut(G) \mid  \| \chi_u - \alpha(\chi_u)\| < 1\}\]
Now, $h(\chi_u^\ast \chi_u) = 1$ and
\[h(\alpha(\chi_u)^\ast\chi_u) =  
\begin{cases}
  1, &\text{if} \;\chi_u=\alpha(\chi_u)\\
  0, &\text{otherwise}
\end{cases}
\]
So, for any $\alpha \in u(\chi_u, 1)$, we have, 
\begin{align*}
  |h(\chi_u^\ast \chi_u) - h(\alpha(\chi_u)^\ast \chi_u)| \leqslant \| \chi_u - \alpha(\chi_u)\| < 1
\end{align*}
 and hence 
\begin{align*}
  |1 - h(\alpha(\chi_u)^\ast \chi_u)| < 1 \Rightarrow \alpha(\chi_u) = \chi_u \Rightarrow \alpha \in K_{[u]}  
\end{align*}

Since $u(\chi_u, 1)$ is an open neighbourhood in $K_{[u]}$ we indeed have that $K_{[u]}$ is an open subgroup of $\Aut(G)$.  
\end{proof}

In fact, much more is true for compact quantum groups which have fusion rules identical to those of connected compact simple Lie Groups. 

\begin{pro}Let $G=(A,\Phi)$ be a compact quantum group having fusion rules identical to those of a connected compact Lie group. Then the group $\mathrm{Out}_{\chi}(G)=\Aut(G)/\Aut_{\chi}(G)$ has finite order. In particular, if $G$ is a $q-$deformation of some simply connected simple compact Lie group, then $\mathrm{Out}_\chi(G)$ has order $1,2,3$ or $6$.   \end{pro}
\begin{proof} Any automorphism $\alpha$ of $G$ induces an order isomorphism of its representation ring $\mathbb{Z}[\widehat{G}]$. And clearly, if $\alpha\in \Aut_{\chi}(G)$, then $\alpha$ induces the trivial isomorphism of its representation ring. So, $\mathrm{Out}_{\chi}(G)$ is easily seen to be a subgroup of the group of order isomorphisms of the representation ring $\mathbb{Z}[\widehat{G}]$. 

But, by results of Handelman\cite{Hand}, it follows that for connected compact groups, the group of order isomorphisms of the representation ring of the group and its outer automorphism group are isomorphic. The proposition now follows from the facts that for connected compact Lie groups, the outer automorphism group is finite and that for simple compact Lie groups, it can only have order $1, 2,$ or $6$.
 \end{proof}
Similar arguments can be used to show that deformations of Hopfian connected compact groups are Hopfian.
\begin{defn} A compact group $G$ is said to be Hopfian if every surjective homomorphism is an isomorphism. In other words, there exists no proper normal subgroup $N$ of $G$ such that $G/N\cong G$.    \end{defn}
Analogously, one can define Hopfian compact quantum groups. 
\begin{defn} A compact quantum group $G=(A,\Phi)$ is said to be Hopfian if every injective quantum homomorphism $\phi :A\rightarrow A$ is also surjective. \end{defn}
\begin{pro} Let $G=(A,\Phi)$ be a compact quantum group which has fusion rules identical to those of a compact connected group $\mathcal{G}$. Suppose that $\mathcal{G}$ is a Hopfian compact group. Then, $G$ is Hopfian as well.  \end{pro}
\begin{proof} Suppose not, then there exists a quantum homomorphism $\alpha :A\rightarrow A$ which is injective but not surjective. Then the induced map of the representation ring $\widehat{\alpha} :\mathbb{Z}[\widehat{G}]\rightarrow \mathbb{Z}[\widehat{G}]$ is injective and order preserving but not surjective.

But $\mathbb{Z}[\widehat{G}]$ is also the representation ring of the compact connected group $\mathcal{G}$. Now the range of $\widehat{\alpha}$ corresponds to a proper subobject of $\widehat{\mathcal{G}}$, which by the Galois correspondence between the subobjects and normal subgroups, corresponds to a proper normal subgroup $\mathcal{N}$ of $\mathcal{G}$ and hence, the representation ring of $\mathcal{G}/\mathcal{N}$ and $\mathcal{G}$ are order isomorphic, and so once again by \cite{Hand}, it follows that $\mathcal{G}$ is isomorphic to $\mathcal{G}/\mathcal{N}$, which is a contradiction.    \end{proof}

\section{Inner Automorphisms and Normal Subgroups I}
In the classical case of compact groups, subgroups are said to be normal if they are stable under all automorphisms of the form $\alpha_s(g) = sgs^{-1}$, $s, g \in G$, a compact group. So, for a compact group $G$, a subgroup $N$ is, by definition, normal if for every automorphism of the form $\alpha_s$, $s \in G$, there exists an automorphism $\beta: N \rightarrow N$ such that: 
 \begin{center}
$\xymatrix{
G \ar[r]^{\alpha_s} & G \\
N \ar@{^{(}->}[u]^{i}\ar[r]^{\beta} & N\ar@{^{(}->}[u]^{i} }$
\end{center}
commutes. 

In fact, more is true: if $N$ is a normal subgroup of $G$, then it corresponds to a unique subobject $\Gamma_N$ of $\widehat{G}$ where $\widehat{G}$ denotes the equivalence classes of all irreducible representations of $G$ \cite{Hew-Ross}. So, it follows that $N$ is always stable under any representation class preserving automorphism (the converse of course is trivial as $\mathrm{Inn}(G) \subseteq \Aut_{\chi}(G)$). In this section, we show that this holds true in the quantum case as well, under certain assumptions. 

Essentially, we want to show that, under certain conditions, given a compact quantum group $G=(A, \Phi)$, $\alpha \in \Aut_{\chi}(G)$ and $N = (B, \Psi)$ a normal subgroup, with $\rho: A \to B$ as the corresponding surjection, there exists a quantum group automorphism $\beta: N \to N$ such that:
\begin{center}
$\xymatrix{
A \ar[d]_{\rho} \ar[r]^{\alpha} & A \ar[d]^{\rho} \\
B \ar[r]_{\beta} & B }$
\end{center} 
commutes. 
\begin{lem}
  Let $G=(A,\Phi)$ be a compact quantum group with $H=(B,\Psi)$ a subgroup of it and $\rho :A\to B$ the corresponing surjection. Let $\alpha :A\rightarrow A$ be an automorphism of G and $\beta :B\rightarrow B$ be a $c^*$-algebraic automorphism such that \\
\centerline{\xymatrix{
A \ar[d]^{\rho} \ar[r]^{\alpha} &A\ar[d]^{\rho}\\
B \ar[r]^{\beta} &B}}\\ 
commutes. Then $\beta$ is also a quantum group automorphism.
\end{lem}
\begin{proof}  
 We want to show that $\Psi \circ \beta = (\beta\otimes\beta)\Psi$. Since $\rho$ is surjective, given any $b\in H$, there exists some $a\in A$ such that $\rho(a)=b$. Now, 
\begin{align*}
(\beta\otimes\beta)\Psi(b)
&=(\beta\otimes\beta)\Psi(\rho(a)) \\
&=(\beta\circ\rho \otimes \beta\circ\rho)\Phi(a)\\
&=(\rho\otimes\rho)(\alpha\otimes\alpha)\Phi(a)\\
&=\Psi(\rho\circ\alpha(a)) \\
&=\Psi(\beta\circ\rho(a))=\Psi(\beta(b))
\end{align*}
and so we are done. 
 \end{proof}
 \begin{pro}
Let $G=(A,\Phi)$ be a compact quantum group and let $H=(B,\Psi)$ be a subgroup of $G$. Let $\alpha$ be an automorphism of $G$ and let $\beta$ be an automorphism of $H$ such that \\
\centerline{\xymatrix{
A \ar[d]^{\rho} \ar[r]^{\alpha} &A\ar[d]^{\rho}\\
B \ar[r]^{\beta} &B}}\\ 
commutes. Then $\alpha :A_{G/H}\rightarrow A_{G/H}$ is a $c^\ast$-algebraic automorphism. Similarly, $\alpha : A_{H\backslash G}\rightarrow A_{H\backslash G}$ is a $c^\ast$-algebraic automorphism.  
  \end{pro}
\begin{proof} 
First we note that $$\rho\circ\alpha = \beta\circ\rho \iff \rho\circ\alpha^{-1}=\beta^{-1}\circ \rho$$
Now, we just have to show that $$\alpha(A_{G/H})\subseteq A_{G/H} $$
This is clear as if $a\in A_{G/H}$, then by definition, $(\rho\otimes \id)\Phi(a)=1\otimes a$. Now,
\begin{align*} 
(\rho\otimes \id)\Phi(\alpha (a))
&= (\rho\otimes \id)(\alpha\otimes\alpha)\Phi(a)\\
&=(\beta\circ\rho\otimes \alpha)\Phi(a)\\
&=(\beta\otimes\alpha)(1\otimes a)\\
&= 1\otimes \alpha(a)
\end{align*}
and so $\alpha (a)\in A_{G/H}$. One can similarly prove the proposition in the case of $A_{H\backslash G}$. \end{proof}
\begin{pro}
 Let $G=(A,\Phi)$ be a coamenable compact quantum group and let $H=(B,\Psi)$ be a subgroup of $G$. Let $\alpha :A\rightarrow A$ be an automorphism of $G$, such that $
\alpha :A_{G/H}\rightarrow A_{G/H}$ is $c^\ast$-algebraic automorphism, then there exists $\beta :B\rightarrow B$ such that $\beta$ is also an automorphism of $H$ and such that\\
\centerline{\xymatrix{
A \ar[d]^{\rho} \ar[r]^{\alpha} &A\ar[d]^{\rho}\\
B \ar[r]^{\beta} &B}}\\ commutes.
\end{pro}
\begin{proof} Let us first note that since $G$ is coamenable, then by Proposition 2.8, $H$ is coamenable as well. Now to prove this proposition, we use Lemma 3.3. So we have to show that $$\rho(a)=0\iff \rho\circ\alpha (a)=0$$
 Suppose that $a\in A_{G/H}$ and $\rho(a)=0$. Then as $\rho(a)=\epsilon_{G}(a)\cdot 1\Rightarrow \epsilon_{G}(a)=0$. 
 
But then 
\begin{alignat*}{3}
&&\epsilon_G (\alpha(a))&=\epsilon_G(a) \\
 \Rightarrow&\mathrm{for}\; a\in A_{G/H},& \rho(a)&=0\\
&\text{if and only if} &\epsilon_G(a)&=0\\
&\text{if and only if} &\rho\circ\alpha(a)&=0
\end{alignat*}
Now, say for some $a\in A$ such that $a>0$, $\rho(a)=0$ and $\rho(\alpha(a))\neq 0$. In that case we have $ h_H(\rho\circ\alpha(a))\cdot 1>0$, where $h_H$ denotes the Haar measure on $H$, which is true since $H$ is coamenable.

It now follows from the equalities $E_{G/H}=(\id\otimes h_H\circ\rho)\Phi$, $(\rho\otimes \rho)\Phi=\Psi\circ \rho$ and $(\id\otimes h_H)\Psi(\cdot)=h_H(\cdot)\cdot 1$ that $ \rho (E_{G/H}\circ\alpha(a))\neq 0 $. But $\alpha$ preserves the Haar measure $h_G$ of G and so does $E_{G/H}$. Now since $\alpha$ is also a $c^\ast$-algebraic automorphism of $A_{G/H}$, we have that $E_{G/H}$ and $\alpha^{-1}\circ E_{G/H}\circ\alpha$ are both $h_G$-preserving conditional expectations onto $A_{G/H}$. Since $h_G$ is faithful, we have by Corollary II.6.10.8 of \cite{Black}, that $$\alpha^{-1}\circ E_{G/H}\circ\alpha = E_{G/H}$$  and hence, $\rho(\alpha(E_{G/H}(a)))\neq 0$.
But as $\rho(E_{G/H}(a))=0$ and $E_{G/H}(a)\in A_{G/H}$, this gives us a contradiction by the first part of the proof and so, there exists a $\beta :B\rightarrow B$ such that  \\
\centerline{\xymatrix{
A \ar[d]^{\rho} \ar[r]^{\alpha} &A\ar[d]^{\rho}\\
B \ar[r]^{\beta} &B}}\\
commutes and is an automorphism of $H$, by Lemma 4.1. \end{proof}

In the case that $N$ is a normal subgroup, we have a more general result. This is because if $\tau=((u^\tau_{ij}))\in M_{d_{\tau}}(A)$ is an irreducible representation of $G$, $d_\tau$ denotes the dimension $\tau$ and $(\tau_{|N},1_N)$ denotes the multiplicity of the trivial representation of $N$ in the restriction of $\tau$ to $N$, then 
$$E_{G/N}(u_{ij}^{\tau})= \begin{cases}
		u_{ij}^{\tau}  & \forall i,j\; \mbox{if } (\tau_{|N},1_N)=d_{\tau}  \\
		0 & \forall i,j\; \mbox{if } (\tau_{|N},1_N)=0
\end{cases} $$
This is true as in the first case, if $(\tau_{|N},1_N)=d_{\tau}$, then $u_{ij}^{\tau}\in A_{G/N}$ for all $i,j$, and in the second case, since the Haar measure of any compact quantum group has the property that it sends the matrix entries of any non-trivial irreducible representation to $0$, we have for all $i,j$, $h_N(\rho(u^{\tau}_{ij}))=0$ as $(\tau_{|N},1_N)=0$ and so, $E_{G/N}(u_{ij}^{\tau})=\sum_{k}h_N(\rho(u^{\tau}_{kj}))u^{\tau}_{ik}=0 $.

Now if $\alpha$ is an automorphism of $G$ such that $\alpha:A_{G/N}\to A_{G/N}$ is a $c^\ast$-algebraic automorphism, we have $$\alpha^{-1} \circ E_{G/N} \circ \alpha = E_{G/N}$$
This is because for any irreducible representation $\tau=((u^{\tau}_{ij}))\in M_{d_{\tau}}(A)$ of $G$, we have that if $(\tau_{|N},1_N)=d_{\tau}$, then $(\alpha(\tau)_{|N},1_N)=d_{\alpha(\tau)}$ and if $(\tau_{|N},1_N)=0$, then $(\alpha(\tau)_{|N},1_N)=0$, where $\alpha(\tau)=((\alpha(u^{\tau}_{ij})))$.

So, by essentially repeating the steps of the proof of the previous proposition, we have the following 
\begin{pro}
  Let $G=(A,\Phi)$ be a compact quantum group and $N=(B,\Psi)$ be a normal subgroup of $G$, with $\rho :A\to B$ the corresponding surjection. Suppose that the Haar measure $h_N$ of $N$ is faithful. Let $\alpha: A \rightarrow A$ is an automorphism of $G$ such that $\alpha: A_{G/N} \to A_{G/N}$ is a $c^\ast$-algebraic automorphism, then there exists $\beta: B \to B$ such that $\beta$ is an automorphism of $N$ and such that:
\begin{center}
$\xymatrix{
A \ar[d]_\rho \ar@{->}[r]^{\alpha} & A \ar[d]^\rho \\
B \ar@{->}[r]_{\beta} & B }$
\end{center} 
commutes.
\end{pro}

\begin{thm}
  Let $G=(A,\Phi)$ be a compact quantum group, $\alpha\in \Aut_{\chi}(G)$ and $N=(B,\Psi)$ be a normal subgroup of it, with $\rho:A\rightarrow B$ the corresponding surjection. Suppose that the Haar measure $h_N$ of $N$ is faithful. Then there exists a $\beta \in \Aut(N)$ such that \\
\centerline{
\xymatrix{
A \ar[d]^{\rho} \ar[r]^{\alpha} &A\ar[d]^{\rho}\\
B \ar[r]^{\beta} &B}
} \\ 
commutes.
\end{thm}
 
\begin{proof}   
This follows from the previous proposition and from the fact that as $\alpha\in \Aut_{\chi}(G)$, $\alpha :A_{G/N}\to A_{G/N}$ is a $c^\ast$-algebraic automorphism.

\end{proof} 

\section{Inner Automorphisms and Normal Subgroups II} 
For this section, we assume that our compact quantum groups are maximal. We start this section by giving a recipe to produce representation class preserving automorphism. 

Let $G=(A,\Phi)$ be a compact quantum group with a non-trivial maximal classical compact subgroup. This subgroup, denoted by $X_G$, consists of the 1-dimensional $c^*-$algebraic representations of the c$^*-$algebra $A$ \cite{Wang-CMP95}. They form a group with product given by $$s_1 \cdot s_2=(s_1\otimes s_2)\Phi$$ and the unit is the counit of $G$, $\epsilon_G$. The surjective map is the obvious evaluation map from $A$,  $$\rho :A\rightarrow C(X_A)$$ $$a\mapsto e_a$$ $\mathrm{where}\;e_a(f)=f(a)$. The inverse may not be explicit, but the fact that $X_G$ is a group follows from Proposition 3.2 of \cite{Maes-VD}, from which it follows that closed sub-semigroups of compact groups, and more generally, of compact quantum groups, are in fact groups. In fact, a compact quantum sub-semi-group of a compact quantum group is itself a compact quantum group.

Let $\Gamma$ be a discrete group and $C^\ast(\Gamma)$ be the full group $c^\ast$-algebra of $\Gamma$. Then for this co-commutative compact quantum group, the maximal classical compact group is $C^\ast(\Gamma/[\Gamma,\Gamma])$, i.e. the full group $c^\ast$-algebra of the abelianisation of $\Gamma$. In case of $SU_{q}(2)$ with $-1<q<1$, the maximal classical compact group is $S^1$, the circle group \cite{Podles}, while in the case of $A_u(n)$, the maximal classical compact group is the unitary group $U(n)$ \cite{Wang-CMP95}.

First we define two automorphisms of $A$, which are just c$^*-$algebraic automorphisms, $$\lambda_s=(s^{-1}\otimes \id)\Phi\; \mathrm{and}\; \rho_s=(\id\otimes s)\Phi$$ Quite clearly $$\lambda_s\rho_t=\rho_t\lambda_s, \forall s,t\in X_G$$
We define $\alpha_s=\lambda_s\rho_s=\rho_s\lambda_S$. Then, $\alpha_s$ is a representation class preserving quantum group automorphism of $G$. Note that we are just inducing the inner automorphism $\alpha_s(x)=s^{-1}xs$ from $X_G$ to the compact quantum group $G$. This class of automorphisms were first defined by Wang in \cite{Wang-PLMS95}.

Now suppose that we have a compact quantum group $G=(A,\Phi)$ and a subgroup of it, $H=(B,\Psi)$. Suppose that the following is true, suppose for each $\tau\in \widehat{G}$ and for all matrix entries, $u_{ij}^{\tau}$, there exists a 1-dimensional $*$-representation $s^{\tau}_{ij} :A\rightarrow \mathbb{C}$ such that $s^{\tau}_{ij}(u^{\tau}_{ij})\neq 0$. We suppose also that for each induced inner automorphism $\alpha_s=\lambda_s \rho_s$ of $G$, there exists an automorphism $\beta :B\rightarrow B$ of $H$ such that \\
\centerline{
\xymatrix{
A \ar[d]^{\rho} \ar[r]^{\alpha_s} &A\ar[d]^{\rho}\\
B \ar[r]^{\beta} &B}
} \\ 
commutes. Then we claim that $H$ is normal. But in this case, $\alpha_s$ maps $A_{G/H}$ into $A_{G/H}$ for each $s\in X_G$. So if $H$ were not normal in $G$, then there would exist an irreducible representation $u^{\tau}=(u^{\tau}_{ij})$ of $G$ such that $0<l<d_{\tau}$, where $l=(u^{\tau}_{|H}, 1_H)$ denotes the multiplicity of the trivial representation of $H$ in the representation $u^{\tau}$ restricted to $H$. We may assume by unitary equivalence that the $l$ trivial representation appear in the upper left diagonal corner of $((\rho(u_{ij}^{\tau})))$. In that case, as has been shown in the proof of Proposition 2.1 of \cite{Wang-JFA}
$$E_{A/H}(u_{ij}^{\tau})= \begin{cases}
	u_{ij}^{\tau}  & \mbox{if } 1\leq i\leq d_{\tau}, 1\leq j\leq l \\
		0 & \mbox{if } 1\leq i\leq d_{\tau}, l<j
\end{cases} $$
 We consider the matrix element $u^{\tau}_{(l+1)1}$. Now, by hypothesis, there exists a $*$-homomorphism $$s^{\tau}_{(l+1)1}: A\rightarrow \mathbb{C}$$ such that $s_{(l+1)1}^{\tau}(u^{\tau}_{(l+1)1})\neq 0$. But 
$$\alpha_{s^{\tau}_{(l+1)1}}(u_{11})=\sum_{m,n} s_{(l+1)1}^{\tau}((u^{\tau}_{m1})^* u_{n1}^{\tau})u_{mn}$$
Now, $u_{(l+1)(l+1)} \not\in A_{G/H}$ as $E_{G/H}(u_{(l+1)(l+1)})=0$. But since $u_{11}\in A_{G/H}$  $\alpha_{s^{\tau}_{(l+1)1}}(u_{11})\in A_{G/H}$. However, $\alpha_{s^{\tau}_{(l+1)1}}(u_{11})$, when expressed in terms of $u^{\tau}_{ij}$, has as coefficient of $u_{(l+1)(l+1)}$, $|s^{\tau}_{(l+1)1}(u^{\tau}_{(l+1)1})|^2$, which is strictly positive. But as $E_{G/H}(u_{(l+1)(l+1)})=0$ and $u^{\tau}_{ij}$'s are linearly independent, this is not possible, and so, $H$ is normal.

So, we have the following 
\begin{pro}
  Let $G = (A, \Phi)$ be a compact quantum group such that for all $u_{ij} \in A$ such that $u_{ij}$ is a  matrix entry for some irreducible representation of $G$, there exists $S: A \to \mathbb{C}$, a one-dimensional representation of $A$ such that $S(u_{ij}) \neq 0$, then $G$ satisfies the following property: 

If $H = (B, \Psi)$ is a quantum subgroup of $G$ such that for any $\alpha \in \Aut_\chi(G)$, there exists $\beta \in \Aut(H)$ such that: 
\begin{center}
$\xymatrix{
A \ar[d]_{\rho} \ar[r]^{\alpha} & A \ar[d]^{\rho} \\
B \ar[r]_{\beta} & B }$
\end{center}
commutes, then $H$ is normal.  
\end{pro}

It is clear that for classical compact groups and for co-commutative quantum groups, the above condition holds, as in the first case, it's just the evaluation at an appropriate point and in the second, the counit suffices. Non-commutative and non-co-commutative examples exist as well (for example, by tensor products). 

\begin{cor}
Let $G=(A,\Phi)$ be a compact quantum group with a non-normal subgroup $H=(B,\Psi)$. Suppose that for each automorphism $\alpha\in \Aut_{\chi}(G)$ there exists an automorphism $\beta$ of $H$ such that the following diagram commutes- 
\begin{center}
$\xymatrix{
A \ar[d]_{\rho} \ar[r]^{\alpha} & A \ar[d]^{\rho} \\
B \ar[r]_{\beta} & B }$
\end{center}
Then there exists an irreducible representation $u^{\tau}$ of $G$ and some matrix entry of it,  $u^{\tau}_{ij}$, such that for any one dimensional $*$-representation of A, $S:A\to \mathbb{C}$, $S(u^{\tau}_{ij})=0$.
\end{cor}
  
We give two such examples, for which there exist non-normal subgroups stabilised by each representation class preserving automorphism: 

\subsection{The $SU_q(2)$ case}  

We show that such is the case for $SU_q(2)$, $-1<q<1,q\neq 0$. We note first that since $SU_q(2)$ has a unique irreducible representation class for a given dimension, any quantum group automorphism is in fact in $\Aut_{\chi}(SU_q(2))$, i.e. $\Aut(SU_q(2))=\Aut_{\chi}(SU_q(2))$. For $SU_q(2)$, the generating unitary representation, denoted by $u$, is the $2\times 2$ matrix \\
\[ \left( \begin{array}{ccc}
\alpha & -q\gamma^*  \\
\gamma & \alpha^*  \\
 \end{array} \right)\] \\
Now, let $\tau:SU_q(2)\rightarrow SU_q(2)$ be an automorphism of $SU_q(2)$. Then, since the irreducible representations $u$ and $\tau(u)$ are in the same representation class, we have for some $((\tau_{ij}))\in U(2)$ 
$$\tau(u)=(((\tau_{ij}))\otimes 1)u(((\overline{\tau_{ji}}))\otimes 1)$$ which tells us that 
$$\alpha\mapsto \tau_{11}\overline{\tau_{11}}\alpha - \tau_{11}\overline{\tau_{12}}q\gamma^*+\tau_{12}\overline{\tau_{11}}\gamma+\tau_{12}\overline{\tau_{12}}\alpha^*$$
$$-q\gamma^*\mapsto \tau_{11}\overline{\tau_{21}}\alpha-\tau_{11}\overline{\tau_{22}}q\gamma^*+\tau_{12}\overline{\tau_{21}}\gamma +\tau_{12}\overline{\tau_{22}}\alpha^*$$
$$\gamma\mapsto \tau_{21}\overline{\tau_{11}}\alpha- \tau_{21}\overline{\tau_{12}}q\gamma^*+\tau_{22}\overline{\tau_{11}}\gamma+\tau_{22}\overline{\tau_{12}}\alpha^*$$
$$\alpha^*\mapsto \tau_{21}\overline{\tau_{21}}\alpha-\tau_{21}\overline{\tau_{22}}q\gamma^* +\tau_{22}\overline{\tau_{21}}\gamma+\tau_{22}\overline{\tau_{22}}\alpha^*$$
Since $\tau$ is $*$-preserving and $\alpha, \gamma, -q\gamma^*, \alpha^*$ are linearly independent, we get, 
$$|\tau_{11}|^2=|\tau_{22}|^2$$
$$q^2 \overline{\tau_{21}}\tau_{12}=\overline{\tau_{21}}\tau_{12}$$
But as $0<q^2<1$, we have $\tau_{21}=\tau_{12}=0$ and $\tau_{11},\tau_{22}\in S^1$, the circle group. So,  
$$\alpha\mapsto \alpha$$
$$\gamma\mapsto \tau_{22}\overline{\tau_{11}}\gamma$$
But $SU_q(2)$ is the universal c$^*$-algebra generated $\alpha,\gamma$ such that 
$$\alpha^*\alpha+\gamma^*\gamma=1,\; \alpha\alpha^*+q^2\gamma\gamma^*=1$$
$$\alpha\gamma=q\gamma\alpha,\; \alpha\gamma^*=q\gamma^*\alpha,\; \gamma\gamma^*=\gamma^*\gamma$$
We see that $\alpha$ and $\tau_{22}\overline{\tau_{11}}\gamma$ also satisfy these relations, which implies that indeed we have an automorphism 
$$\tau:SU_q(2)\rightarrow SU_q(2)$$
$$\alpha\mapsto \alpha$$
$$\gamma\mapsto \tau_{22}\overline\tau_{11}\gamma$$
where $\tau_{11},\tau_{22}\in S^1$, the circle group. It is easily checked that this automorphism is also a quantum group automorphism by verifying the relation on the generators.

However, since $\tau_{22} \overline{\tau_{11}} \in S^1$, there exists some $\kappa \in S^1$ such that $\overline{\kappa^2}=\tau_{22} \overline{\tau_{11}}$ and so, by \cite{Wang-PLMS95}, we get that the same automorphism is induced by the matrix 
\[
\begin{bmatrix}
\kappa & 0 \\
0 & \overline{\kappa}
\end{bmatrix} 
\]
and this comes from the induced automorphism of the maximal compact group $S^1$ of $SU_q(2)$. But, for these automorphisms, as they are induced from the quantum subgroup $S^1$ of $SU_q(2)$, we then have the following: 

\begin{thm}
  For any $\alpha \in \Aut(SU_q(2)) = \Aut_\chi(SU_q(2))$, with $0<q^2<1$, we have the following commutative diagram- \\
\centerline{
\xymatrix{
SU_q(2) \ar[r]^{\alpha} \ar[d]_{\rho}&SU_q(2) \ar[d]^{\rho} \\
C(S^1) \ar[r]_{id}& C(S^1)  
} }
\end{thm}

\subsection{The $A_u(n)$ case}
In case of $A_u(n)$, we have a surjective homomorphism 
$$\phi :U(n)\to \Aut_{\chi}(A_u(n))$$
$$t\mapsto \phi_t$$
where $\phi_t :A_u(n)\to A_u(n)$ is defined by the property that 
$$\phi_t(u)=(t\otimes 1)u(t^\ast\otimes 1)$$
where $u=((u_{ij}))\in M_n(A_u(n))$ is the fundamental unitary. This, by the universal property of $A_u(n)$, extends to an automorphism of $A_u(n)$. Surjectivity of $\phi$ follows as for any $\alpha \in \Aut_{\chi}(A_u(n))$, there exists some $t\in U(n)$ such that 
$$\alpha(u)=(t\otimes 1)u(t^\ast\otimes 1)\Rightarrow \alpha=\phi_t$$ 
\begin{thm}
For any $t\in U(n)$, the following diagram commutes-\\
\centerline{
\xymatrix{
A_u(n) \ar[r]^{\phi_t} \ar[d]_{\rho}&A_u(n) \ar[d]^{\rho} \\
C(U(n)) \ar[r]_{\psi_t}& C(U(n))
}}
where $\rho$ denotes the canonical surjection onto $C(U(n))$, the algebra of complex-valued complex functions of the group $U(n)$, which is the maximal compact subgroup of $A_u(n)$ and $\psi_t$ denotes the automorphism of $C(U(n))$, induced by the inner automorphism 
$$\beta_t:U(n)\to U(n),\;\; s\mapsto tst^{\ast}$$   
\end{thm} 
\begin{proof} This is easily shown by checking the relation on $u_{ij}$'s, the matrix entries of $u$, the fundamental unitary of $A_u(n)$. \end{proof}

We have the following corollary, which follows easily as the diagram of the previous proposition commutes.

\begin{cor} For the homomorphism $\phi:U(n)\to \Aut_{\chi}(A_u(n))$, we have $$\mathrm{ker}\:\phi=\mathrm{Z}(U(n))=\{\lambda I\;:\;\lambda\in S^1\}$$ Also, each non-trivial $\phi_t$ is a $c^\ast$-algebraic outer automorphism of $A_u(n)$. \end{cor}

We now want to show that $U(n)$ is not normal in $A_u(n)$.

\begin{lem}
  Let $G$ be a compact quantum group with subgroups $N_1$ and $N_2$. Suppose that $N_2$ is a subgroup of $N_1$ and that $N_2$ is normal in $G$. Then, $N_2$ is normal in $N_1$. 
\end{lem}
\begin{proof}
  This follows easily from Lemma 2.1 of \cite{Fima}. See also the proof of Theorem 5.17 of \cite{Pinz}.
\end{proof}

\begin{lem}
Let $G=(A,\Phi)$, $N=(B,\Psi)$ and $H=(C,\Delta)$ be compact quantum groups such that $H$ is a subgroup of $N$ and $N$ is a subgroup of $G$, so $H$ is also a subgroup of $G$. Suppose further that $N$ and $H$ are normal in $G$. Then the compact quantum group $G/H$ has $N/H$ as a normal subgroup with the quotient being $G/N$.
\end{lem}
\begin{proof}
It follows from the previous lemma that $H$ is normal in $N$. We have the three corresponding surjections 
$$\rho_0 :A\to C$$
$$\rho_1 :A\to B$$
$$\rho_2 :B\to C$$
such that $\rho_0=\rho_2\circ \rho_1$. 

Now, $\rho_1(A_{G/H})=B_{N/H}$ as if $a\in A_{G/H}$ then
\begin{align*}
(\id \otimes\rho_2)\Psi(\rho_1(a)) &=(\rho_1\otimes\rho_2\circ\rho_1)\Phi(a)\\
&=( \rho_1\otimes \id)(\id\otimes \rho_0)\Phi(a)\\
&= \rho_1(a)\otimes 1
\end{align*}
and so $\rho_1(A_{G/H})\subseteq B_{N/H}$. But then, if $b\in B_{N/H}$, and if $x\in A$ is such that $\rho_1(x)=b$, then for $E_{G/H}(x)=(\id\otimes h_H\circ\rho_0)\Phi(a)$, we have
\begin{align*} 
\rho_1(E_{G/H}(x))&=(\id\otimes h_H\circ\rho_2)(\rho_1\otimes\rho_1)\Phi(x)\\
&=( \id\otimes h_H\circ\rho_2)\Psi(\rho_1(x))=b
\end{align*}
So, we have that $N/H$ is indeed a subgroup of $G/H$ with the surjection being the map $\rho_1$ restricted to $A_{G/H}$. As $A_{G/N}\subseteq A_{G/H}$, it is easily seen that $N/H$ is normal in $G/H$, and that $(G/H)/(N/H)=G/N$.
\end{proof}
\begin{thm}
The subgroup $C(U(n))$ is not normal in $A_u(n)$.\end{thm}
\begin{proof} Since $C(S^1)$ is normal in $A_u(n)$, as is shown in Proposition 4.5 of \cite{Wang-JFA}, and also in $C(U(n))$, we have using the previous lemma and Theorem 1 of \cite{Chiv} that $C(U(n))$ is not normal in $A_u(n)$ \end{proof}

We end this section by computing the group $$\mathrm{Out}_{\chi}(A_u(n))=\Aut(A_u(n))/\Aut_{\chi}(A_u(n))$$
\begin{pro} $Aut(A_u(n))/Aut_{\chi}(A_u(n))=\mathbb{Z}_2$ \end{pro}
\begin{proof} Since $A_u(n)$ has exactly 2 irreducible representation of dimension $n$, and the automorphism $\gamma$ defined by 
$$u=(u_{ij})\mapsto \overline{u}= (u^*_{ij})$$
is a quantum group automorphism, the proposition follows from the fact that any automorphism not in $Aut_{\chi}(A_u(n))$ will map 
$$u\mapsto (t\otimes 1)\overline{u}(t^*\otimes 1), t\in U(n)$$,
which is composition of two quantum group automorphisms 
$$u\xrightarrow{\gamma} \overline{u}\xrightarrow{\gamma_t} t\overline{u}t^*$$
But $\gamma_t$ is in $Aut_{\chi}(A_u(n))$ and hence the result follows. 
\end{proof}

\section{Central Subgroup and Center} 
In case of compact groups, we have that $\Inn(G) = G/Z(G)$, where $Z(G)$ denotes the center of $G$. In this section, we try to identify the center of a given compact quantum group.
\begin{defn}
[\cite{Wang-JFA}]
 A subgroup $H=(B,\Psi)$ of a compact quantum group $G=(A,\Phi)$ is said to be central if $(\rho\otimes \id)\Phi = (\rho\otimes \id)\sigma\Phi$, where $\rho:A\to B$ denotes the corresponding surjection and $\sigma$ the flip map on $A\otimes A$.\end{defn}
\begin{pro}
  A central subgroup of a compact quantum group is always co-commutative.
\end{pro}
\begin{proof} We have that to show that $\Psi(a)=\sigma\Psi(a)$ for all $a\in H$.

Let $s\in A$ such that $\rho(s)=a$, then we have
\begin{align*} 
\Psi(\rho(s))&=(\rho\otimes\rho)\Phi(s)\\
&= (\id\otimes \rho)(\rho\otimes \id)\Phi(s)\\
&= (\id\otimes \rho)(\rho\otimes \id)\sigma\Phi(s)\\
\sigma(\rho\otimes\rho)\Phi(s) &=\sigma\Psi(\rho(s))
\end{align*}
Hence, we are done.  \end{proof}
It follows easily from the definitions that central subgroups of compact quantum groups are always normal.
\begin{thm}
 Let $H=(B,\Psi)$ be a subgroup of a compact quantum group $G=(A,\Phi)$. Then $H$ is a central subgroup of $G$ if and only if given any irreducible representation $u^{\tau}$ of $G$, there exists a unique 1-dimensional representation $\lambda_n$ of $H$ such that $u^{\tau}$ restricted to $H$ decomposes as a $d_{\tau}$ sum of $\lambda_n$, where $d_{\tau}$ denotes the dimension of $u_{\tau}$. In other words, 
$$ (u^{\tau}_{|H},\lambda_n)=d_{\tau}  $$ 
\end{thm}
\begin{proof} 
$(\Leftarrow)$ This is be easily seen by a straightforward calculation as $$(\rho\otimes \id)\Phi(u^{\tau}_{ij})=(\rho\otimes \id)\sigma\Phi(u^{\alpha}_{ij})$$ for matrix elements of all irreducible representations.\\
$(\Rightarrow)$ $H$ is central, so by the previous lemma, $H$ is co-commutative.      
But then any irreducible representation $(u^{\tau})$ when restricted to $H$ decomposes as a sum of 1-dimensional representations of $H$. By unitary equivalence, we may assume that $\rho(u_{ij}^{\tau})=0$ if $i\neq j$ and $\rho(u_{11}^{\tau})=\lambda_1$ and $\rho(u_{ii}^{\tau})=\lambda_i$ for some $i\neq 1$. We want to show that $\lambda_1=\lambda_i$. Now,
\begin{align*} 
(\rho\otimes \id)\Phi(u^{\tau}_{1i})
&=(\rho\otimes \id)\sum_{k}u^{\tau}_{1k}\otimes u_{ki}^{\tau} \\
&=\sum_{k}\rho(u_{1k}^{\tau})\otimes u_{ki}^{\tau} 
=\lambda_1\otimes u^{\tau}_{1i}
\end{align*}
and
\begin{align*} 
(\rho\otimes \id)\sigma \Phi(u_{1i}^{\tau})
&=(\rho\otimes \id)\sum_k u_{ki}^{\tau}\otimes u_{1k}^{\tau}\\
& =\sum_k \rho(u_{ki}^{\tau})\otimes u^{\tau}_{1k}
= \lambda_i \otimes u_{1i}^{\tau}
\end{align*}
and hence, $\lambda_1=\lambda_i$.
\end{proof}

\begin{cor}
  Suppose $G = (A, \Phi)$ is a compact quantum group and $H = (B, \Psi)$ is a central subgroup of it. Suppose $\alpha \in \Aut_\chi(G)$. Then,
\begin{center}
$\xymatrix{
A \ar[d]_{\rho} \ar[r]^{\alpha} & A \ar[d]^{\rho} \\
B \ar[r]_{id} & B }$
\end{center}   
commutes.
\end{cor}
\begin{proof}
  Since $\alpha \in \Aut_\chi(G)$, for any irreducible representation $u^{\tau}$ of $G$, with dimension $d_{\tau}$, 
  $$\alpha(u^{\tau})=(t\otimes 1)u(t^{\ast}\otimes 1)$$ 
  for some $t$ in $U(d_{\tau})$. The result now follows by direct calculation, using the previous theorem and assuming, by unitary equivalence, that given an irreducible representation $u^{\tau}$:
  $$\rho(u_{ij}^{\tau})= \begin{cases}
		0  & \text{if } i\neq j \\
		\lambda^{\tau} & \text{if } i=j
	
\end{cases}$$
where $\lambda^{\tau}$ is a 1-dimensional representation of $H$.
\end{proof}

\begin{pro}
  Let $G=(A,\Phi)$ be a compact quantum group and let $X_G$ be its maximal classical compact subgroup. Let $X_0:=\{s\in X_G : \alpha_s = id \;\mathrm{on}\; A\}$, i.e. the set of 1-dimensional $*-$representations of $A$ such that the associated induced inner automorphisms of $G$ is trivial. Then $X_0$ is a subgroup of $X_A$ and is a central subgroup of $G$. 
\end{pro}
\begin{proof} 
Its easy to check that $X_0$ is a subgroup of $X_A$. 

Now since 
\begin{alignat*}{4}
&&\alpha_s&=\lambda_s\rho_s&\\
&\Rightarrow& \alpha_s&=\id \Leftrightarrow \lambda_{s^{-1}}=\rho_s&\\
&\Rightarrow& (s\otimes \id)(\rho\otimes \id)\Phi &= (\id\otimes s)(\id\otimes \rho)\Phi,\;&\forall s\in X_0\\
&\Rightarrow& (s\otimes \id)(\rho\otimes \id)\Phi &= (s\otimes \id)(\rho\otimes \id)\sigma\Phi\\
&\Rightarrow& (\rho\otimes \id)\Phi&=(\rho\otimes \id)\sigma\Phi&
\end{alignat*}
and hence, $X_0$ is central in $G$.
 \end{proof}
 
Let $G=(A,\Phi)$ be a compact quantum group, and $\widehat{G}$ be the object corresponding to equivalence classes of its irreducible representations. Following \cite{Mcmullen}, we give the following 
\begin{defn} Let $\Sigma\subseteq \widehat{G}$ be a subobject. For $a,b\in \widehat{G}$, we define $a\sim b$ if and only if $a\times \bar{b}\cap \Sigma\neq \phi$. Here $\bar{b}$ denotes the conjugate of $b$ and $a\times b$ denotes the set of all elements of $\widehat{G}$ representatives of which are subrepresentations of $u^a\otimes u^b$ where $u^a, u^b$ are irreducible representations of $G$ and $[u^a]=a$ and $[u^b]=b$.  \end{defn}
Using Propn 3.2 of \cite{Podles-W}, its easy to check that this defines an equivalence relation. We call the equivalence classes of this relation as $\Sigma-$cosets.

Obviously, the set-wise product of 2 $\Sigma-$ cosets is a union of $\Sigma-$cosets.

\begin{defn} We say a subobject $\Sigma\subseteq \widehat{G}$ is a central subobject if the $\Sigma-$cosets form a group. In this case, the product of two $\Sigma-$cosets is itself a $\Sigma-$coset. $\Sigma$, which is itself a $\Sigma-$coset, acts as the identity element. The group is denoted as $\widehat{G}/\Sigma$. \end{defn}

\begin{pro}
Let $G=(A,\Phi)$ be a compact quantum group and $H=(B,\Psi)$ a normal subgroup. Let $\Sigma$ denote the subobject of $\widehat{G}$ corresponding to equivalence classes of irreducible representations of $G$ that decompose as direct sum of trivial representation when restricted to $H$. Then $H$ is central if and only if $\Sigma-$cosets form a group.
\end{pro}

\begin{proof} $(\Rightarrow)$ Let $H$ be central. By Theorem 6.3, we know that for any irreducible representation $u^{\tau}$ of $G$, there exists a unique 1-dimensional representation $\phi^{\tau}\in \widehat{H}$ such that $(u^{\tau}_{|H},\phi^{\tau})=d_{\tau}$, where $d_{\tau}$ denotes the dimension of $u^{\tau}$. We consider the following map-
$$\pi : \widehat{G}\mapsto \widehat{H}$$
$$[u^{\tau}]\mapsto \phi^{\tau}$$
where $[u^{\tau}]$ denotes the equivalence class of $u^{\tau}$ in $\widehat{G}$.
The map $\pi$ is then easily checked to be multiplicative, in the sense that if we take two irreducible representations of $G$, $u^{\alpha}$ and $u^{\beta}$, then $$\sigma\in [u^{\alpha}]\times [u^{\beta}]\mapsto \phi^{\alpha}\cdot\phi^{\beta}$$

Now, $\Sigma$ consists of $[u^{\tau}]$ such that $\pi([u^{\tau}])=1_H$, so this implies that $$\alpha\sim\beta\Leftrightarrow \pi(u^{\alpha})=\pi(u^{\beta})$$
and hence, cosets correspond to preimages of elements of the group $\widehat{H}$ and since $\pi$ is multiplicative, we have that $\widehat{G}/\Sigma$ is indeed a group.\\
$(\Leftarrow)$ Let $H$ not be central, then by Theorem 6.3, there exists some irreducible representation $u^{\sigma}$ of $G$, with $[u^{\sigma}]\in \widehat{G}$ denoting its equivalence class, such that $u^{\sigma}_{|H}=\oplus_{i=1}^{n}\xi_{i}$ upto equivalence, where either for some $i$, dim$(\xi_i)>1$ or, in case all $\xi_i$ have dimension 1, for some $i\neq j, \xi_i\neq\xi_j$, where $\xi_i$'s are irreducible representations of $H$. Then we have that 
$$u^\sigma\otimes\bar{u^\sigma}=\oplus_{i,j=1}^{n}(\xi_i\otimes \bar{\xi_j})$$ upto equivalence.
But then, in either of the aforementioned cases, $1_H$ and some other non-trivial representation appear in the decomposition of $(u^{\sigma}\otimes \bar{u^{\sigma}})_{|H}$ into irreducible representations of $H$. And so if we let $[\sigma]$ denote the $\Sigma$-coset corresponding to $[u^{\sigma}]$ and $[\bar{\sigma}]$ the $\Sigma$-coset corresponding to $[\bar{u^{\sigma}}]$, $[\sigma]\times[\bar{\sigma}]$ is a union of more than than coset, which gives a contradiction. \end{proof}
\begin{pro} Let $\mathcal{Z}:=\{\Sigma\subseteq \widehat{G}: \Sigma$ is a central subobject $\}$. Let $$\widetilde{\Sigma}:= \cap_{\Sigma\in \mathcal{Z}}\Sigma$$ Then $\widetilde{\Sigma}\in \mathcal{Z}$ \end{pro}
\begin{proof} Let $[a]$ and $[b]$ be two $\widetilde{\Sigma}-$cosets, we want to show that $[a]\times[b]$ is also a $\widetilde{\Sigma}-$coset. This follows easily from the following two facts -: \\
\item(i) If $a$ and $a_1$ belong to the same $\widetilde{\Sigma}-$coset, then they belong to the same $\Sigma -$coset for all $\Sigma\in Z$, obvious as $\widetilde{\Sigma}\subseteq \Sigma$ for all $\Sigma\in \mathcal{Z}$.\\
\item(ii) If $a$ and $a_1$ belong to the same $\Sigma-$coset for every $\Sigma\in \mathcal{Z}$ then $a$ and $a_1$ belong to the same $\widetilde{\Sigma}-$coset. This is true because for $\Sigma\in \mathcal{Z}$, the $\Sigma-$cosets form a group. Now, since $a$ and $a_1$ belong to the same $\Sigma-$coset, we have that $$a\times \bar{a_1}\cap \Sigma\neq \phi$$
But then as $\Sigma-$cosets form a group, we have that $[a]\times [\bar{a_1}]=\Sigma$, where $[a]$ denotes the $\Sigma$-coset corresponding to $a$. This implies that for any $\sigma \in a\times\bar{a_1}, \sigma \in \Sigma$ for all $\Sigma\in \mathcal{Z}$, so $\sigma\in \widetilde{\Sigma}$ and hence, $a$ and $a_1$ are in the same $\widetilde{\Sigma}-$coset. \end{proof}

We call $\widetilde{\Sigma}\subseteq \widehat{G}$ the center subobject of $G$. This corresponds to the center of the compact quantum group $G=(A,\Phi)$. The subalgebra in $A$ generated by $\widetilde{\Sigma}$ gives us the quotient quantum group $G/Z(G)$.

Let $G=(A,\Phi)$ be a compact quantum group. We assume that $G$ is maximal. Let $\Sigma \subseteq \widehat{G}$ be a central object and $H_\Sigma= \widehat{G}/\Sigma$ be the group of $\Sigma$-cosets. We, then have: 
\begin{pro}
  $C^\ast(H_\Sigma)$ is a central subgroup of $G$, with $\Sigma$ being the subobject of $\widehat{G}$ cosisting of equivalence classes of irreducible representations of $G$ that decompose as a sum of trivial representation when restricted to it, so its left and equivalently right, coset space is generated by matrix entries of the representatives of elements of $\Sigma$ as a subalgebra of $A$.  
\end{pro}

\begin{proof}
We refer to the Preliminaries section and to \cite{Woro-Invent} for definitions of unexplained terms.

We have for the compact quantum group $G=(A, \Phi)$, a triple $$(R_G, \{\mathcal{H}_\alpha\}_{\alpha \in R_G}, \{\Mor(\alpha, \beta)\}_{\alpha,\beta\in R_G})$$ By Proposition 3.3 of \cite{Woro-Invent} for any model , $(M, \{V_r\}_{r \in R_G})$, there exists a $\ast$-homomorphism $$\varphi_M: \mathcal{A} \to M$$ which extends to $A$ as $G$ is maximal. 

We want to give a model with $M = C^\ast(H_\Sigma)$. First, let us take a complete set of irreducible representations $\{\alpha_k\}_{k \in I}$ of $G$. So, $\alpha_k$'s are representatives of elements of $\widehat{G}$ and belong to $R_G$. And, they are complete in the sense that, the map $$\kappa:\{\alpha_k\}_{k \in I} \subseteq R_G \to \widehat{G}$$ $$\alpha_k \mapsto [\alpha_k]$$ is  one-one and onto.

For these $\alpha_k \in R_G$ such that $\alpha_k \in B(\mathcal{H}_k) \otimes A$, we define $$V_{\alpha_k} = 1_{\mu_k} \otimes \delta_k \in B(\mathcal{H}_k) \otimes C^{\ast}(H_{\Sigma})$$ Here, $\delta_k$ is the element of the group $H_\Sigma$  corresponding to the $\Sigma$-coset that contains $[\alpha_k]$. Now, for any $r \in R_G$, we know that there exists a finite set $\{\alpha_{k_1}, \dots, \alpha_{k_n}\}\subseteq \{\alpha_k\}_{k \in I}$ such that there exist $P_{k_i} \in \Mor(\alpha_{k_i}, r)$ with the properties that: 
\begin{align*}
\sum P_{k_1} P_{k_i}^\ast &= I_{\mathcal{H}_r} \\
P_{k_i}^\ast P_{k_i} &= I_{\mathcal{H}_k} \\
\text{and } P_{k_i}P_{k_i}^\ast P_{k_j}P_{k_j}^\ast& = 0 \text{ for all } i\neq j \in \{1, \dots, n \} 
\end{align*}
We then define $V_r \in B(\mathcal{H}_r) \otimes C^\ast(H_\mathcal{E})$ , 
\[V_r := \sum_i P_{k_i}P_{k_i}^\ast \otimes \delta_{k_i}.\]
It is then straightforward to check that:
\[V_r \otop  V_s = V_{rs}\] and \[V_r(t \otimes 1) = (t \otimes 1) V_s\] for all $r, s \in R_G, t \in \Mor(s, r)$. So, we have a model for $$\{R_G, \{\mathcal{H}_\xi\}_{\xi \in R_G}, \Mor(\alpha, \beta)_{\alpha,\beta \in R_G}\}$$ and so a $\ast$-homomorphism: 
\[\varphi_{\Sigma}: A \to C^\ast(H_\Sigma)\]
It follows easily that $\varphi_{\Sigma}$ is a quantum group homomorphism and also that $C^\ast(H_{\Sigma})$ is a central subgroup of $G$ with $\Sigma\subseteq \widehat{G}$ as the corresponding subobject.      
\end{proof}
So, associated to any central subobject $\Sigma\subseteq \widehat{G}$, there exists a central subgroup of $G$, with the algebra generated by representatives of elements of $\Sigma$ giving its left/right coset space. In particular, for $\widetilde{\Sigma}$, we have the group $H_{\widetilde{\Sigma}}$, which can be regarded as the center of the compact quantum group $G = (A, \Phi)$. 
\begin{lem}
  Suppose we have compact quantum groups $G=(A,\Phi)$, $N_1=(B_1,\Psi_1)$, $N_2=(B_2,\Psi_2)$ and we further assume that all are maximal. Let $N_1$ and $N_2$ be normal subgroups of $G$ such that $\Sigma_{N_1} \subseteq \Sigma_{N_2}$ where $\Sigma_{N_i}$ is the subobject of $\widehat{G}$ corresponding to $N_i$ (and hence, consisting of equivalence classes of irreducible representations of $G$ that are direct sums of the trivial representation when restricted to $N_i$), $i=1,2$ . Then, $N_2$ is a normal subgroup of $N_1$.  Further, if the corresponding surjections are denoted $$\rho_i: A \to B_i,\; i =1, 2$$ then, the surjection $$\rho_0: B_1 \to B_2$$ satisfies:
\[\rho_2 = \rho_0 \circ \rho_1\]  
\end{lem}
\begin{proof}
  We have that $\left.\rho_1\right|_{\mathcal{A}} = \mathcal{B}_1$ and  $\left.\rho_2\right|_{\mathcal{A}} = \mathcal{B}_2$. We define:
\[\rho_0: \mathcal{B}_1 \to \mathcal{B}_2\]
by $$\rho_0(\rho_1(a)) = \rho_2(a)\; \mathrm{for}\; a \in \mathcal{A}$$ 

This is well-defined, as $\ker(\rho_1) \subseteq \ker(\rho_2)$, which follows from Lemma 4.4 of \cite{Wang-JFA}. It is then easy to check that $\rho_0$ is a $\ast$-homomorphism, so can be extended to $B_1$. It is also surjective and is in fact a quantum group homomorphism. Normality follows from Lemma 5.6.
\end{proof}
So, it follows that any central subgroup of $G$ is in fact a subgroup of the center of that $G$. 

\begin{rem} 
Let $G=(A,\Phi)$ be a maximal compact quantum group. We consider the set of discrete groups
$$\bar{Z}(G):= \{F: C^{*}(F)\;\mathrm{is\; a\; central\; subgroup\; of\; G}\}$$ 
We say that $F_0\leq F_1$ if there exists a surjective map $$\rho_{0}^{1}:F_1\rightarrow F_0$$ such that the induced map $\widetilde{\rho_{0}^{1}} : C^{*}(F_1)\rightarrow C^{*}(F_0)$ has the property that \\
\centerline{\xymatrix{
A \ar[rd]^{\rho_0} \ar[r]^{\rho_1}
&C^{*}(F_1) \ar[d]^{\widetilde{\rho^{1}_{0}}}\\
&C^{*}(F_0)}}\\
commutes.

This then gives us an inverse system of discrete groups. One can directly show this without appealing to the previous proposition. It is easy to see that if we have $F_0\leq F_1\leq F_2$, with respective maps, $$\rho_{0}^{1}:F_1\rightarrow F_0$$ $$\rho_{1}^{2}:F_2\rightarrow F_1$$ and $$\rho_{0}^{2}:F_2\rightarrow F_1$$ then $$\rho_{0}^{2}=\rho_{0}^{1}\circ \rho_{1}^{2}$$
Now given $F_1$ and $F_2$ in $\bar{Z}(G)$, we want to show that there exists $F_0$ in $\bar{Z}(G)$ such that $F_1\leq F_0$ and $F_2\leq F_0$. We have surjections $$\rho_1 :A\rightarrow C^{*}(F_1)$$ $$\rho_2 :A\rightarrow C^{*}(F_2)$$
We consider the following map \\ 
\centerline{\xymatrixcolsep{3pc}\xymatrix{
A \ar[r]^-{\widetilde{\Phi}} &A\otimes_{\max} A \ar[r]^-{\rho_1\otimes_{\max}\rho_2} &C^*(F_1)\otimes_{\max}C^{*}(F_2)\cong C^{*}(F_1\times F_2)\\
}}\\ \\
Here $\widetilde{\Phi}$ denotes the extension of the map 
$$\Phi:\mathcal{A}\subseteq A\rightarrow \mathcal{A}\otimes_{\text{alg}}\mathcal{A}\subseteq A\otimes_{\max} A$$
We have to check that $(\rho_1\otimes_{\max}\rho_2)\widetilde{\Phi}$ is a quantum group homomorphism and that it is central. This is easy to check as $(F_1,\rho_1)$ and $(F_2,\rho_2)$ are central subgroups, so it is easily checked for matrix entries of irreducible representations of $G$.

However, the map need not be surjective. But we can consider its range, which is a cocommutative quantum group. Let $$\mathrm{Ran}(\rho_1\otimes_{\max}\rho_2)=C^{*}(F_0)$$
We now wish to show that $F_0\geq F_1$ and $F_0\geq F_2$. Let us show that $F_0\geq F_1$. So, we want to show that \\
\centerline{\xymatrix{
A \ar[rd]^{\rho_1} \ar[r]^{\Psi_0}
&C^{*}(F_0) \ar[d]^{\pi_1}\\
&C^{*}(F_1)}}\\
commutes. Here, $$\Psi_0=(\rho_1\otimes_{\max}\rho_2)\widetilde{\Phi}$$ and $$\pi_1:C^{*}(F_0)\rightarrow C^{*}(F_1)$$ is the restriction to $C^{*}(F_0)$ of the map 
$$\widetilde{\pi_1}:C^{*}(F_1\times F_2)\rightarrow C^*(F_1)$$ $$\delta_{g_1}\otimes\delta_{g_2}\mapsto \delta_{g_1}$$
Let $(u^{\tau}_{kl})$ be matrix entries of some arbitrary irreducible representation of $A$ such that 
$$\rho_2(u^{\tau}_{kl})= \begin{cases}
		\delta_g  & \text{if } k=l, g\in F_2 \\
		0 & \text{if } k\neq l
\end{cases}$$
This is possible as $F_2$ is a central subgroup of $(A,\Phi)$, invoking Theorem 6.3.
Then we have 
\begin{align*}
\pi_1(\rho_1\otimes\rho_2)\Phi(u^{\tau}_{kl})
=\pi_1(\rho_1\otimes \rho_2)\sum_j u_{kj}^{\tau}\otimes u_{jl}^{\tau}
=\pi_1(\rho_1(u_{kl}^{\tau})\otimes \delta_{g})
\end{align*}
But since $$\rho_1(u^{\tau}_{kl})=\lambda\cdot \delta_{g_1}$$ for some $g_1\in F_1$ and $\lambda\in \mathbb{C}$, we have that $$\pi_1(\rho_1(u_{kl}^{\tau})\otimes \delta_g)=\rho_1(u^{\tau}_{kl})$$ and so we have $F_0\geq F_1$ and similarly, $F_0\geq F_2$.

So, we have an inverse directed system and taking inverse limit, we get a discrete group $F=\varprojlim F_i$. This is of course the center of the compact quantum group $G$, as can be seen using Proposition 6.10. 
 \end{rem}

\section{Center Calculations}
We calculate the center of some compact quantum groups. The following theorem is helpful in many cases: 
\begin{thm}
Compact quantum groups having identical fusion rules have isomorphic center.    
\end{thm}
\begin{proof}
  Let $G_1 = (A, \Phi)$ and $G_2 = (B, \Psi)$ be two compact quantum groups with $\widehat{G_1}$ and $\widehat{G_2}$, the set of equivalence classes of irreducible representations respectively. Then, as $G_1$ and $G_2$ have identical fusion rules, hence, there exists a bijection: 
\[p: \widehat{G_1} \to \widehat{G_2}\]
such that for all $v, s \in \widehat{G_1}$, we have, $$p(v \times s) = p(v) \times p(s)$$ as subsets of $G_2$. Thus, $\Sigma \subseteq G_1$ is central if and only if $p(\Sigma) \subseteq \widehat{G_2}$ is central, as follows easily from Proposition 6.8. So, centers of $G_1$ and $G_2$ are isomorphic. 
\end{proof}
Thus, we have the following: 
\begin{enumerate}
\item The center of $SU_q(2)$ is $\mathbb{Z}_2$ for $-1 \leqslant q \leqslant 1$ and $q \neq 0$.
\item $SO_q(3)$ has trivial center, $-1 \leqslant q \leqslant 1$, and $q \neq 0$. 
\item $C^\ast(\Gamma)$, for any discrete subgroup $\Gamma$, has as center $\Gamma$. 
\item $B_u(Q)$ has the same fusion rules as $SU(2)$ \cite{Ban96} and so its center is $\mathbb{Z}_2$, where $Q$ is a $n \times n$ matrix with $Q \overline{Q}=cI_n'$. 
\item For $B$ a finite dimensional $c^\ast$-algebra with $\dim(B) \geq 4$ and $\tau$ the canonical trace on it, the compact quantum group $A_{\text{aut}}(B, \tau)$ \cite{Wang-CMP98} has the same fusion rules as $SO(3)$ \cite{Ban99} and so has trivial center.  
\end{enumerate}
\begin{pro}
  Let $G=(A, \Phi)$ be a compact quantum matrix group with $u$ being its fundamental representation. Assume that $u$ is irreducible, then $Z(G)$, the center of $G$, is either $\mathbb{Z}/n\mathbb{Z}$ for some $n \in \mathbb{Z}$ or $\mathbb{Z}$.   
\end{pro}
\begin{proof}
  By Theorem 6.4, and the fact that $u$ is the fundamental representation and is irreducible, it follows that if $Z(G) =\Gamma$, then, $\Gamma$ is generated by $\delta_g$ where $$\rho(u) = \underbrace{\delta_g \oplus \dots \oplus \delta_g}_\text{dim(u)-times}$$
  where $\rho$ denotes the surjection from $G$ onto its center. Hence, $\Gamma$ is a quotient of $\mathbb{Z}$. 
\end{proof}
\begin{cor}
The compact quantum group $A_u(Q)$, $Q \in GL_n(\mathbb{C})$ has center $\mathbb{Z}$. \end{cor} 
\begin{proof}  
It is shown in \cite{Wang-JFA} that $A_u(Q)$ has $C(S^1)=C^{\ast}(\mathbb{Z})$ as a central subgroup. The result now follows from the previous lemma, as $u$, the fundamental representation of $A_u(Q)$, is irreducible.
\end{proof}

The notion of the chain group $c(G)$ for a given compact group $G$ was defined by Ba\"umgartel and Lledo \cite{Baum-L} and was shown by M\"uger in \cite{Muger} to be isomorphic to $\widehat{Z(G)}$, the dual group of the center $Z(G)$ of $G$. We prove that the same is true for compact quantum groups. 

\begin{defn}Given a compact quantum group $G=(A, \Phi)$, we define the chain group  $$c(G):=\widehat{G}/\sim_1$$ to be the group of equivalence classes of the equivalence relation $\sim_1$ defined for $X,Y\in \widehat{G}$ as follows : $$X \sim_1 Y$$ if and only if $$\exists n \in \mathbb{N}\; \mathrm{and}\; z_1,\dots, z_n \in \widehat{G}$$ such that $$ X \in z_1 \times \dots \times z_n \;\mathrm{and}\; Y \in z_1 \times \dots \times z_n$$ \end{defn} 
Let $$\mathcal{E}_z:=\{z\in \widehat{G}:z\sim_1 [1_G]\}$$
where $[1_G]\in \widehat{G}$ denotes the equivalence class of $1_G$, the trivial representation of $G$. It is then, straightforward to see that $\mathcal{E}_z$ is a subobject of $\widehat{G}$. 
\begin{pro}
  $\mathcal{E}_z$ is a central suboject of $\widehat{G}$, and $\widehat{G}/\mathcal{E}_z \cong \widehat{G}/\sim_1 \cong c(G)$.
\end{pro}
\begin{proof}This will follow if we can show that, for $a, b \in \widehat{G}$, $a \sim b$ in the sense of Definition 6.6, with $\Sigma=\mathcal{E}_z$ if and only if $a \sim_1 b$. \\
$(\Rightarrow)$ Let for $a,b\in \widehat{G}$, $a\sim b$ in the sense of Definition 6.6. Then we have $a \times \overline{b} \cap \mathcal{E}_z \neq \emp$, so there exists $z_1, \dots, z_n\in \widehat{G}$ such that for some $k \in a \times \overline{b}$, we have $$k\in z_1 \times \dots \times z_n$$ and $$[1_G] \in z_1 \times \dots \times z_n$$ This implies $$[1_G] \in z_1 \times \dots \times z_n \times a \times \overline{b}$$ and $$a \times \overline{b} \subseteq z_1 \times \dots\times z_n \times a \times \overline{b}$$ So, $a\sim_1 b$ as $a\in z_1 \times \dots\times z_n \times a \times \overline{b}\times b$ and $b\in z_1 \times \dots\times z_n \times a \times \overline{b}\times b$\\\\
$(\Leftarrow)$ Let $a\sim_1 b$ for $a,b\in \widehat{G}$. Then $a\in z_1 \times \dots \times z_n$ and  $y\in z_1 \times \dots\times z_n$ for some $z_1, \dots, z_n \in \widehat{G}$. But then obviously,  $$1\in z_1 \times \dots \times z_n\times \overline{b}$$ and  $$a \times \overline{b}\subseteq z_1 \times \dots \times z_n \times \overline{b}$$
So, we have that $a\sim b$ in the sense of Definition 6.6 with $\Sigma=\mathcal{E}_z$. 
\end{proof}

Now to show that the chain group is indeed isomorphic to the center of the compact quantum group $G$, we have to show $$\widetilde{\Sigma}=\mathcal{E}_z$$ as subsets of $\widehat{G}$, where $\widetilde{\Sigma}$ denotes the center subobject of $G$. But by the previous proposition we have $\widetilde{\Sigma}\subseteq \mathcal{E}_z$. So, we want to show that $\mathcal{E}_z\subseteq \widetilde{\Sigma}$ which is true if and only if $\mathcal{E}_z\subseteq \Sigma$ for any central subobject $\Sigma$. This can be easily shown as if $a\in \mathcal{E}_z$, then there exist $z_1,...,z_n$ in $\widehat{G}$ such that 
$$1\in z_1\times z_2\times ... \times z_n$$ and
$$a\in z_1\times z_2\times ... \times z_n$$
But for any central subobject $\Sigma$, the $\Sigma$-cosets form a group, denoted by $H_{\Sigma}$, by Proposition 6.8. Now consider the product 
$$[z_1][z_2]...[z_n]$$
in $H_{\Sigma}$. This is the identity element of $H_{\Sigma}$. But then $a\in \Sigma$. 
So, we indeed have that the chain group of $G$ is isomorphic to the center of $G$.

\appendix 

\section{The Commutator Subgroup}
In this appendix, we briefly discuss the commutator subgroup of a compact quantum group.

For  classical compact groups, the subobject corresponding to the commutator subgroup is the one consisting of equivalence classes of all $1$-dimensional representations of the group. 

In case of a compact quantum group $G = (A, \Phi)$, it is easy to see that if we take the $c^\ast$-subalgebra $A_{\ab} \subseteq A$, generated by the $1$-dimensional representations of $G$, then $A_{\ab}$ is a  $c^\ast$-algebraic completion of $\mathbb{C}[G_{\ab}]$, where $G_{\ab}$ denotes the group generated by the $1$-dimensional representations, i.e. the group-like elements, of $G$. 

It is also easy to see that any quantum group homomorphism, 
\[\widetilde{\rho}: C^\ast(\Gamma) \to A\]
for some discrete group $\Gamma$, has image contained in $A_{\ab}$. So, this is considered to be the abelianisation of $G$. 

However, in contrast to the classical case, the subalgebra corresponding to the abelianisation of $G$ may fail to be the coset space of a normal subgroup of $G$.

Let $p\in\mathcal{A}\subseteq A$. Then the right adjoint action of $p$ on $A$ is given by 
$$p\curvearrowright q\mapsto \sum_{i}S(p^{(1)}_i)qp^{(2)}_i$$
where $\Phi(p)=\sum_i p^{(1)}_i\otimes p^{(2)}_i$ and $S$ is the antipode of $G$. Similarly, the left adjoint action of $p$ will take $q$ to $$\sum_i p^{(1)}_i qS(p^{(2)}_i)$$
It is not hard to see that given any subgroup of $G$, the left coset space (resp. right coset space) is invariant under the right adjoint action (resp. left adjoint action) of any $p\in \mathcal{A}$. 

In particular, the coset space of any normal subgroup of $G$ is necessarily invariant under the left as well as the right adjoint action of any $p\in \mathcal{A}$ (if a subalgebra $A_1$ of $A$ is invariant under the left (resp. right) adjoint action of all $p\in \mathcal{A}$, we say that $A_1$ is invariant under the left (resp. right) adjoint action of $G$). But, as has been observed in Remark 1.2.1 of \cite{Chiv}, the converse is also true.

Let $G_1=(A_1,\Phi_1)$ and $G_2=(A_2,\Phi_2)$ be two compact quantum groups. Wang in \cite{Wang-CMP95} constructed a compact quantum group structure on $A_1\ast A_2$, the full free product of $A_1$ and $A_2$, such that the embeddings of the $c^\ast$-algebras $A_1$ and $A_2$ in $A_1\ast A_2$ are quantum group morphisms and classified its irreducible representations in terms of those of $G_1$ and $G_2$. 

\begin{pro} Assuming that $G_1$ and $G_2$ are non-trivial compact quantum groups, the subalgebras $A_1$ and $A_2$ are not invariant under the right (or left) adjoint action of $G_1\ast G_2$. \end{pro}
\begin{proof} The intuition comes from the co-commutative case, as when $G_1$ and $G_2$ are co-commutative, the result follows from the fact that if $\Gamma_1$ and $\Gamma_2$ are non-trivial discrete groups, then in $\Gamma=\Gamma_1\ast \Gamma_2$, neither of the $\Gamma_i, i=1,2$, is a normal subgroup.

More generally, let $G_1$ and $G_2$ be any two non-trivial compact quantum groups. Let ${}^1u^{\alpha}_{ij}\in A_1$ be a matrix entry of a representative of some $\alpha\in \widehat{G_1}$ and let ${}^2u^{\beta}_{kl}$ be the matrix entry of a representative of some $\beta \in \widehat{G_2}$. Then for ${}^1u^{\alpha}_{ij},{}^2u^{\beta}_{kl}\in A_1\ast A_2$ 
\begin{equation}{}^2u^{\beta}_{kl}\curvearrowright {}^1u^{\alpha}_{ij}\mapsto \sum_m ({}^2u^{\beta}_{mk})^\ast ({}^1u^{\alpha}_{ij})({}^2u^{\beta}_{ml})\end{equation}
Now ${}^1u^{\alpha}_{ij}\in A_1\subseteq A_1\ast A_2$ but by Theorem 3.10 of \cite{Wang-CMP95}, the right hand side of (A.1) is a sum of matrix entries of some irreducible representation of the compact quantum group $G=G_1\ast G_2$, so it follows from the linear independence of matrix entries of irreducible representations that $A_1$ is not invariant under the right adjoint action of $G$. 

Similarly, one can prove the result for $A_2$ and for left adjoint action.  
    \end{proof}
Now, taking $G_1$ to be a compact quantum group with no non-trivial $1$-dimensional representation and $G_2=C^\ast(\Gamma)$, for some discrete group $\Gamma$, it follows from Theorem 3.10 of \cite{Wang-CMP95} that for $G=G_1\ast G_2$, $G_{\ab}=\Gamma$, so $A_{\ab}=C^\ast(\Gamma)$. But by the previous proposition, $A_{\ab}$ cannot be the coset space of some normal subgroup of $G$ as it is not closed under the right adjoint action of $G$. 
\begin{rem} Interestingly, the free product construction is also used in \cite{Pinz} to show that the identity component is not normal in general, unlike the classical case.  \end{rem}

We end by noting a deviation from the classical theory in case of $A_u(Q), Q\in GL(n,\mathbb{C})$. 

For compact connected Lie groups, it is known that the intersection between the center and the commutator subgroup is finite. In particular, if the commutator subgroup is the whole group, then the center is finite and conversely, if the center is finite, the commutator subgroup is the whole group. Such Lie groups are called semisimple.

However, in contrast to the classical case, we have, for $A_u(Q)$, which is connected, the commutator subgroup is the whole of $A_u(Q)$, as the only one dimensional representation is the trivial representation \cite{Wang-JOT} but the center, $\mathbb{Z}$, is infinite. 

\section*{Acknowledgements}
The author wishes to thank his supervisor A. Mohari for constant encouragement, P. Sankaran, K.N. Raghavan, P. Chatterjee, A. Prasad and P.S. Chakraborty for discussions and comments, K. Rajkumar, K. Mukherjee and S. Guin for patient hearing and words of encouragement, K. Sampath for all the help with typesetting.

Several ideas of this paper became clear when the author attended the conference Non-Commutative Geometry and Quantum Groups (NCGQG) 2012 at Oslo, the author wishes to thank the organisers E. Bedos, M. Landstad, N. Larsen, S. Neshveyev and L. Tuset for financial support to attend the conference. Thanks are also due to all participants of this conference for several fruitful conversations. The author especially wishes to thank C. Pinzari and Y. Arano.  

Last but not the least, the author thanks all members of his institute, in particular the library staff, for creating a wonderful atmosphere for research.

\end{document}